\documentclass{amsart}
\usepackage{a4wide,amsmath,amssymb}
\usepackage{url}
\usepackage{geometry}

\numberwithin{equation}{section}
\usepackage{mathtools}
\usepackage{hyperref}

\usepackage{xcolor}

\usepackage{xcolor}
\usepackage{cases}

\newtheorem{thm}{Theorem}[section]
\newtheorem{lemma}[thm]{Lemma}
\newtheorem{prop}[thm]{Proposition}
\newtheorem{cor}[thm]{Corollary}
\theoremstyle{remark}
\newtheorem{rem}[thm]{Remark}
\theoremstyle{definition}

\newcommand{\R}{\mathbb{R}}

\newcommand\intO{\iint_{Q_T}}

\makeatletter
\@namedef{subjclassname@2020}{\textup{2020} Mathematics Subject Classification}
\makeatother

\newcommand{\aac}{\`a}

\newcommand{\dys}{\displaystyle}

\newcommand{\eps}{\varepsilon}

\DeclareMathOperator*{\esssup}{ess\,sup}

\title[]{On maximal regularity estimates for quasilinear evolution equations via the integral Bernstein method}
\date{\today}

\author{Alessandro Goffi}
\address{Dipartimento di Matematica ``Tullio Levi-Civita'', Universit\`a degli Studi di Padova, 
via Trieste 63, 35121 Padova (Italy)}
\curraddr{}
\email{alessandro.goffi@unipd.it}
\thanks{}
\author{Tommaso Leonori}
\address{Dipartimento di Scienze di Base ed Applicate per l'Ingegneria, Sapienza Universit{\aac} di Roma, Via Antonio Scarpa 10, 00161 Roma (Italy)}
\curraddr{}
\email{tommaso.leonori@uniroma1.it}
\thanks{}
\thanks{The authors are members of the Gruppo Nazionale per l'Analisi Matematica, la Probabilit\`a e le loro Applicazioni (GNAMPA) of the Istituto Nazionale di Alta Matematica (INdAM) and were partially supported by the INdAM-GNAMPA Project 2023 ``Problemi variazionali/nonvariazionali: interazione tra metodi integrali e principi del massimo''. A. Goffi was partially supported by the King Abdullah University of Science and Technology (KAUST) project CRG2021-4674 ``Mean-Field Games: models, theory and computational aspects" and by the project funded by the EuropeanUnion – NextGenerationEU under the National Recovery and Resilience Plan (NRRP), Mission 4 Component 2 Investment 1.1 - Call PRIN 2022 No. 104 of February 2, 2022 of Italian Ministry of University and Research; Project 2022W58BJ5 (subject area: PE - Physical Sciences and Engineering) ``PDEs and optimal control methods in mean field games, population dynamics and multi-agent models". This research was partially carried out while A. Goffi was Postdoctoral  research fellow at Dipartimento di Scienze di Base ed Applicate per l'Ingegneria, Sapienza Universit{\aac} di Roma.}

\begin{document}

\maketitle

\begin{abstract} 
This work addresses the problem of (global) maximal regularity for quasilinear evolution equations with sublinear gradient growth and right-hand side in Lebesgue spaces, complemented with Neumann boundary conditions. The proof relies on a suitable variation of the Bernstein technique and the Bochner identity, and provides new results even for the simpler parabolic $p$-Laplacian equation with unbounded source term. As a byproduct we also obtain a second-order estimate that can be of independent interest when the right-side of the equation belongs to $L^m$, $m\neq 2$. This approach leads to new results even for stationary problems.
\end{abstract}

\noindent

\section{Introduction}

We investigate in this paper gradient regularity properties of solutions to the   Neumann boundary-value problem of the type 
\begin{equation}\label{pp}
\begin{cases}
\partial_t u-\mathrm{div}(|Du|^{p-2}Du) =H(x,t,Du) &\quad\text{ on }Q_T,\\
\partial_\nu u=0&\quad\text{ on } \partial\Omega\times(0,T),\\
u(x,0)=u_0(x)&\quad\text{ on } \Omega\ .
\end{cases}
\end{equation}
Here, $p>1$, $\Omega$ is a bounded convex set in $\R^N$ of class $C^1$, $N\geq1$ being the dimension of the ambient space, and $T>0$ is the time horizon of the problem, $\partial_\nu$ is the outward normal derivative, and  $Q_T:=\Omega\times(0,T)$. \\
The nonlinearity $H(x,t,\xi):Q_T\times\R^N\to\R$ and $u_0:\R^N\to\R$ are given, while  $u:Q_T\to\R$ is the unknown. The function $H$ is often called the Hamiltonian and will be of the form
\begin{equation}\label{modelH}
H(x,t,\xi)=|\xi|^\gamma+f(x,t),
\end{equation}
where
\begin{equation}\label{growth}
0\leq \gamma   <  \ell \qquad \mbox{ where } \quad \ell := \max\left\{\frac{p}2,p-1-\frac{p-2}{N+2} \right\}
\end{equation}
and
\[
f\in L^m(Q_T)\qquad \text{ or } \qquad f\in L^m(0,T;W^{1,m}(\Omega)),\ m>2\,, 
\]
and 
$u_0 \in W^{1,r}(\Omega)$ for some $r \geq 1$, 
depending on the type of result we are interested in. The papers \cite{Guibe,Magliocca, PorzioDCDS} explain in detail the motivation of identifying $\ell$,  defined in \eqref{growth},  as the parabolic counterpart of the linear growth exponent $\gamma=p-1$ appearing in the stationary framework. 

Since $\gamma$ is allowed to be chosen $0$, we   include some new results, and recover some other known estimates by different proofs as a byproduct, for the more classical time-dependent $p$-Laplacian equation, even if we  consider slightly more general diffusions (see Section \ref{sec;assres} for more details on the assumptions).\\
The first result shows that $W^{1,m}$ regularity is preserved along the (nonlinear) flow driven by the operator $\partial_t-\Delta_p$. More precisely, we show for  $\gamma=0$ that 
\[
\|Du\|_{L^\infty_t(L^m_x)}\leq C(m,T)(\|Df\|_{L^m_{x,t}}+\|Du_0\|_{L^m}),\quad m\in(2,\infty).
\]
Notably, $C(m,T)\to \overline{C}$ as $m\to\infty$, thus showing the preservation of Lipschitz regularity by different methods than those arising from the theory of viscosity solutions \cite{BDL}. Furthermore, the estimate is stable in the limit $p\to1$: this provides a parabolic counterpart of an estimate found in \cite{PorrCCM} for the correspondent  class of elliptic problems. The precise statement is in Theorem \ref{conserv}.

\medskip 

The second main result, see Theorems \ref{mainintro1} and \ref{mainintro2}, treats functions $H$ of the form \eqref{modelH}--\eqref{growth}. The model results we prove are the following: 
\begin{equation}\label{main1}
\|Du\|_{L^q(Q_T)}\leq C_1+C_2\|f\|_{L^m(Q_T)}^{\frac{1}{p-1-\frac{p-2}{N+2}}},
\end{equation}
where
\[
\max\left\{ 2, \frac{Np+4}{N(p-1)+2}\right\}<m<N+2\ , \quad   \quad q=\frac{(N+2)(p-1)-(p-2)}{N+2-m}m,
\]
and
\begin{equation}\label{main2}
\|Du\|_{L^r(Q_T)}\leq C_1+C_2\|f\|_{L^{N+2}(Q_T)}^{\frac{1}{p-1-\frac{p-2}{N+2}}}\ , \qquad r<\infty .
\end{equation}
Note that 
\[
q\to \frac{(N+2)m}{N+2-m}\quad \text{ as } \quad p\to 2,
\]
which is the classical parabolic Sobolev exponent corresponding to the anisotropic space $W^{2,1}_m$ associated to the heat operator. Besides, we point out that the above estimates are stable as $\gamma\to0$. The first estimate \eqref{main1} extends previous qualitative bounds found in Theorem 1.9 of \cite{BDGOjfa} to a wider regime of summability (see Remark \ref{comparison}). In quantitative terms, both estimates can be seen as the parabolic counterpart of those obtained in Theorem 4.3-(i) and (ii) of \cite{CianchiMazyaJEMS} for stationary equations without lower order perturbations. Furthermore, as a byproduct our results provide the second-order estimate
\[
  \||Du|^\omega D u\|_{L^2(0,T;W^{1,2}(\Omega))}\leq C_1+C_2\|f_\eps \|_{L^{m}(Q_T)}^\frac{Nm}{N+2-m},\quad \mbox{ with } \quad  \omega =\frac{Nm (p-1)-(p-2) (m-2)}{2 (N+2-m)}-1 \ .
\]
Such a bound agrees (at least formally) in the limit $m\to2$ with the level of regularity found in \cite{CianchiMazyaJGeom} for the case $m=2$ in the context of Dirichlet boundary conditions (see Theorem 2.2 therein). To our knowledge, this is the first instance of a second order estimate when $f\in L^m$, $m\neq 2$, in the parabolic setting. We mention that our approach provides new results even in the stationary framework, see Theorem \ref{mainapp}, where some contributions are already available \cite{Lou,Mercuri,Montoro,Mosconi}, and complement those in \cite{CGL} obtained in the superlinear regime.

 We conclude the treatment by raising the integrability of the gradient up to $L^\infty$ when $m>N+2$ without using nonlinear Calder\'on-Zygmund type estimates: this is done in Theorem \ref{mainintro3} combining the Bernstein method and the Stampacchia approach for zero-th order estimates, and this appears to be new in the context of eikonal type nonlinearities. This also provides an approach to increase the summability of the gradient from $L^r$ to $L^\infty$ for Hamilton-Jacobi type equations ($p=2$) driven by diffusive operators that do not satisfy the Calder\'on-Zygmund theory. These bounds were  typically achieved by different methods that combine a variational Bernstein argument with an iteration \`a la De Giorgi-Moser \cite{Dib,DiBFriedman}, see Remark \ref{p<2}.\\

Let us mention that, despite the huge  effort devoted to analyze gradient regularity properties of solutions to nonlinear degenerate/singular parabolic equations and systems, cf. e.g. \cite{EvansAlikakos,AcerbiMingioneDuke,BaroniHabermann,BaroniLindfors,Dib,FrehseSchwarzacher} and the references therein, few results in the setting of maximal regularity are known even for the mere parabolic $p$-Laplacian. Most of the literature deals with proving pointwise gradient estimates and local gradient boundedness in the borderline regime given by the Lorentz class $L(N+2,1)$, cf. \cite{MingioneMemoirs,KuusiMingioneSNS,KuusiMingioneMatAnn,KuusiMingioneARMA}, or to zero-th order estimates \cite{PorzioOrsina,PorzioDCDS,Magliocca}. In addition, the research concerning second order regularity is still limited to few results, see e.g. \cite{Beirao,DiBFriedman,Parviainen,Dib} and the references in \cite{CianchiMazyaJGeom}, where restrictions on $p$ or higher regularity on the data are sometimes required. We refer to the recent paper \cite{ACCFM} and \cite{CianchiMazyaARMA2,CianchiMazyaJMPA,Dong,CianchiMazyaCPAA,GigaTsubouchi,MingioPalaSurvey} for further local and global regularity properties of stationary quasilinear equations.\\
 
Our technique revolves around an integral version of the Bernstein method. The main idea behind the (classical) Bernstein technique is that if some derivatives of the unknown function are subsolutions to an elliptic/parabolic equation, then the maximum principle provides a priori estimates for the derivatives of such solutions: we refer to the seminal paper by S. Bernstein \cite{Bern} (see also  \cite{Serrin2,Lions80} for more developments). The presence of integrable data in the equation, however, rules out the possibility of using sup-norm techniques \cite{CLS,LPcpde,VeronJFA,AttouchiSouplet,Attouchi}, possibly embodied within the theory of viscosity solutions \cite{BDL,CDLP}, and some variation of the method involving integral identities is needed.  The integral Bernstein method was introduced almost simultaneously by E. DiBenedetto-A. Friedman \cite{DiBFriedman} and P.-L. Lions in \cite{Lions85}, see also \cite{BardiPerthame, LasryLions}, to study quite different problems. The former analyzes gradient boundedness and finer properties of solutions to certain classes of nonlinear degenerate and singular parabolic equations and systems. The latter involves the $L^q$ integrability of the gradient of solutions to the stationary Hamilton-Jacobi equation
\[
-\Delta u+|Du|^\gamma=f(x),
\]
whose peculiarities are the superlinear character $\gamma>1$ of the first-order term and the $L^q$ integrability of $f$. The recent paper \cite{CGell} refined the method to study maximal regularity in $L^q$ spaces of solutions to the above equation, see also \cite{CCV,GoffiPediconi} for the Neumann case, \cite{Gccm} for Dirichlet problems and \cite{CV} for interior bounds. The authors together with M. Cirant developed a variation of the Bernstein technique in \cite{CGL} to analyze the problem of maximal gradient regularity for the stationary counterpart of \eqref{pp} under the ``superlinear'' growth condition $\gamma>p-1$. Gradient regularity along with comparison and uniqueness principles in the sublinear regime is the matter of e.g. \cite{BettaCPAA,DNFG}. Other recent variations of the integral Bernstein method appeared in the context of elliptic and parabolic systems with superquadratic gradient terms arising from stochastic control and differential games \cite{BensoussanFrehseCPAA,Urbano}. \\
 Gradient estimates for parabolic viscous equations with first-order nonlinearities having superlinear growth are more delicate, and have been tackled using different methods based on duality techniques: this is the matter of the papers \cite{c22par, cg20,CGpar,Gjee,Gccm}. Indeed, the Bernstein argument developed in \cite{CGell} fails in general for time-dependent problems having  superlinear behavior with respect to the gradient. Though the picture is completely understood for linear diffusions $p=2$ with unbounded source terms from the results in \cite{c22par, CGpar}, the quasilinear parabolic case seems ruled out from these works. Zero-th order estimates at the level of maximal regularity have been studied in \cite{Magliocca}, see also \cite{GMP14} for the elliptic counterpart. Few results are available regarding the gradient regularity of solutions to time-dependent quasilinear problems with right-hand side in Lebesgue spaces, see \cite{DiBFriedman,Dib} and \cite{PorzioDCDS}. Interior and global H\"older estimates for the larger class of semisolutions in the supernatural case $\gamma>p$ can be found in \cite{Gprsa}, see also \cite{CDLP,DP} for $p=2$. In this direction, we also mention those in \cite{BDL,LPcpde} for stationary equations and Lipschitz data, and the earlier paper \cite{EstebanMarcati,Laurencot} which concerned regularizing effects related to semiconcavity-type estimates. 
 
Here, to reach our maximal regularity results,  we develop a version of the integral Bernstein method suitable to quasilinear evolution equations in the sublinear regime of the nonlinearity. Remarkably, in the borderline case $\gamma=0$ this provides new results even for the parabolic $p$-Laplace equation with unbounded source terms. 
Our argument is based on a delicate interplay between a nonlinear integral $p$-Bochner identity, cf. Proposition \ref{bochner}, refined parabolic Sobolev inequalities and simpler tools such as the H\"older and Young inequalities. Even though the sublinear framework can be considered as a perturbative regime, and maximal regularity estimates along with gradient boundedness in the case $p=2$ usually follow by standard Calder\'on-Zygmund regularity, cf. \cite{CKS}, we stress that estimates in the case of nonlinear diffusions cannot be immediately deduced as a consequence of the purely diffusive case.  For this reason our analysis takes on a stronger meaning: it provides a self-contained treatment of the gradient regularity for the whole integrability range of $m$ (above the duality exponent) without using Calder\'on-Zygmund type theory.

\section{Assumptions and main results}\label{sec;assres}
We suppose in the whole paper the following set of assumptions on the diffusion: let $\alpha: \R^+ \to \R^+$ be a $C^2$ function that satisfies 
\begin{equation} 
\label{Aa}
\begin{array}{c}
\displaystyle -1<i_\alpha= \inf_{s>0}\frac{2 s\alpha'(s)}{\alpha(s)}\leq \sup_{s>0}\frac{2s\alpha'(s)}{\alpha(s)}=s_\alpha<\infty\,,
\\[2.0 ex]
\displaystyle 
\exists  \, 0< \underline{\alpha} \leq \overline{\alpha} <\infty \, : \qquad   \underline{\alpha} s^{\frac{p-2}{2}}\leq\alpha(s)\leq     \overline{\alpha}  s^{\frac{p-2}{2}}\,,\quad \forall s>0\,, p>1.
\end{array}
\end{equation}

Notice that \eqref{Aa} implies  the existence of a constant $\widetilde{c}_\alpha = 1+i_\alpha>0$ such that

\begin{equation}\label{a'a}
 2 s\alpha'(s)+\alpha(s) 
 \geq \widetilde{c}_\alpha \alpha(s)\ ,\quad \forall s>0\,.
\end{equation}

Since our results  deal with a priori estimates on solutions to \eqref{pp}, we first regularize the principal part of the operator considering  

\[
\forall \eps >0 \quad \underline{\alpha} (s+\eps)^{\frac{p-2}{2}} \leq \alpha_\eps(s) :=  \alpha (s + \eps)  \leq \overline{\alpha} (s+\eps)^{\frac{p-2}{2}}\ ,\quad \forall s>0\,.
\]
We also recall that, at least for smooth functions (say $C^2$)  $z$ we can approximate the $p$-Laplacian of $z$ as
\[
\lim_{\eps\to0} \ \alpha_\eps(|Dz|^2)\mathcal{A}_{\eps,z}(D^2z) \ = \ \alpha (|Dz|^2)\mathcal{A}_{0,z}(D^2z)  \ = \ \Delta_p z\,,
\]
where 
for every $v\in C^1$, $M\in\mathcal{S}_N$, the space of $N\times N$ symmetric matrices, we set
\[
\mathcal{A}_{\eps,v}(M)=\mathrm{Tr}(M)+    \frac{ 2   \alpha_\eps ' (|Dv|^2)}{\alpha_\eps (|Dv|^2) }  \ M  Dv  Dv.
\]
Here and in the rest of the paper, we denote by $w= |D u|^2$ so that by   the identity $Dw=\frac12D^2uDu$ we can rephrase the previous equality as 
$$
\mathcal{A}_{\eps, u} (D^2 u )
=
  \Delta u +      \frac{ 2   \alpha_\eps ' (w)}{\alpha_\eps (w) }  D^2 u Du Du   
=
  \Delta u +      \frac{     \alpha_\eps ' (w)}{\alpha_\eps (w) }  Dw  Du   \,.
$$

Notice that for all $\eps >0$,  $-\alpha_\eps(|Du|^2)\mathcal{A}_{\eps,z}(D^2u)$ defines a quasilinear operator which is smooth and uniformly elliptic. 
As usual, we consider the following approximation to \eqref{pp} with the problem
\begin{equation}\label{approxp}
\begin{cases} 
\partial_t u_\eps-\alpha_\eps(|Du_\eps|^2)\mathcal{A}_{\eps,u_\eps}(D^2u_\eps)=H_\eps (x,t, Du_\eps) \qquad & \mbox{in }  Q_T  \,,\\
\partial_\nu u_\eps= 0  & \mbox{on } \partial \Omega\times (0,T) \,,\\
u_\eps (x,0) = u_0 (x) & \mbox{in } \Omega\,,\\
\end{cases} 
\end{equation}
where $H_\eps$ is a smooth approximation of $H$. 

Our goal is to derive estimates for  $u_\eps$ that are independent from $\eps$.

\bigskip
The first result shows the preservation of Sobolev (up to Lipschitz) regularity for equations without gradient dependent nonlinearities. It also contains a kind of continuous dependence estimate with respect to the data of the problem, which is stable in the limits $p\to1$ and $m\to\infty$.

\begin{thm}\label{conserv}
Let $\Omega$ be convex and $u_\eps$ be a solution to \eqref{approxp}, with $H_\eps (x,t, Du_\eps)=f_\eps(x,t)\in L^m(0,T;W^{1,m}(\Omega))$, and  $u_0\in W^{1,m}(\Omega)$, $m>2$. Then, for every $m\in(2,\infty]$ we have the a priori estimate
\[
\esssup_{t \in [0,T]} \|Du_\eps (t)\|_{L^m(\Omega)}\leq C(m,T)(\|D f_\eps \|_{L^m(Q_T)}+\|Du_0\|_{L^m(\Omega)}),
\]
where $C(m,T)\to \overline{C}>0$ as $m\to\infty$. 
\end{thm}
The second result we want to prove treats more general quasilinear equations with right-hand side in $L^m(Q_T)$. It shows a maximal gradient regularity for integrable source terms in the regime $2<m<N+2$ and slowly increasing gradient terms, namely having a \lq\lq sublinear\rq\rq{ } growth. This complements earlier results in \cite{Magliocca} where zero-th order estimates were addressed for a similar problem equipped with Dirichlet boundary conditions, and those in \cite{BDGOjfa,PorzioDCDS} concerning maximal regularity estimates below the duality exponent of the right-hand side.
\begin{thm}\label{mainintro1}
Let $p>1$, $\Omega $ be convex,  
\[
H_\eps (x,t, \xi)= (|\xi|^2+\eps)^{\frac{\gamma}2} + f_\eps (x,t)\qquad \text{ with } \quad  0\leq \gamma   <  \ell
\]
and 
$$
f_\eps (x,t ) \in L^{m} (Q_T)\, , \quad \mbox{  with } \quad m_p<m<N+2\,, \quad \mbox{ where  } \quad m_p := \max\left\{ 2, \frac{Np+4}{N(p-1)+2} \right\}\,.
$$ 
If $u_0 \in W^{1, \rho} (\Omega)$,  with $ \rho = N\frac{(p-1)m-(p-2)}{N+2-m}$,  then  there exists $C$ depending on 
$\Omega, N, T,p ,\gamma, \alpha, m$ such that any   solution $u_\eps$ to \eqref{approxp} satisfies  
\begin{equation}\label{e1}
   \|D u_\eps\|_{L^{q} (Q_T)} \leq C \Big(1+ \|f_\eps \|_{L^{m}(Q_T)}^\frac{1}{p-1- \frac{p-2}{N+2}}+ \|Du_0\|_{L^\rho(\Omega)}^{\frac{m(p-1)-(p-2)}{m(p-1)-2\frac{p-2}{N+2}}} \Big)\,, \qquad \mbox{ with } \qquad 
   q = \frac{(N+2)(p-1)m-(p-2)m}{N+2-m}
\end{equation}
and
\begin{equation}\label{e2}
   \| D u_\eps \|_{L^\infty(0,T;L^\rho(\Omega))} \leq C\Big(1+C \|f_\eps \|_{L^{m}(Q_T)}^{  {\frac{m}{m(p-1)-(p-2)}}}+ \|Du_0\|_{L^\rho(\Omega)}
 \Big)  .
\end{equation}
Moreover 
\begin{equation} \label{so}
\|\ |D u_\eps|^{ \omega } D u_\eps \|_{ L^{2} (0,T; H^1 (\Omega))}  
 \leq   C\Big(1+ \|f_\eps \|_{L^{m}(Q_T)}^\frac{Nm}{N+2-m}+ \|Du_0\|_{L^\rho(\Omega)}^{\frac{\rho}{2}}\Big)  \, \qquad \mbox{ with } \qquad   \omega =  \frac{Nm (p-1)-(p-2) (m-2)   }{2(N+2-m)}-1
 \,.
\end{equation}
\end{thm}

\begin{rem} 
Let us observe that (formally) for $m\to 2^+$ we have that $\omega \to p-2$ and $\rho\to p$. This recovers the same level of regularity found in \cite[Theorem 2.2]{CianchiMazyaJGeom} for $H\equiv 0$ (for problems with Dirichlet boundary conditions). It is worth mentioning the second order estimate found in Theorem 2.4 of \cite{DiBFriedman} for homogeneous equations and systems.
\end{rem}
\begin{rem}
The second estimate in Theorem \ref{mainintro1} in the mixed Lebesgue space $L^{\infty}_t (L^{\rho}_x)$ shows a preservation of regularity with respect to the initial datum $u_0\in W^{1,\rho}$. On the contrary,   {since $q>\rho$ in the regime $p>\frac{2N}{N+m}$}, the first estimate (which displays the same summability in the space-time cylinder) shows a spatial improvement with respect to the regularity of the initial trace, at the expenses of loosing some integrability in the time variable.
\end{rem}

In the limiting case $m=N+2$ we get the following result. 
\begin{thm}\label{mainintro2}
Let $p>1$, $\Omega $  be convex,  
\[
H_\eps (x,t, \xi)= (|\xi|^2+\eps)^{\frac{\gamma}2} + f_\eps (x,t)
\qquad 
 \mbox{ with  } \quad 0\leq \gamma  < \ell
\,,  \quad \mbox{ and } \quad f(x,t ) \in L^{N+2} (Q_T)\,. 
\]
If $u_0 \in W^{1, \rho} (\Omega)$,  $\forall \rho >1$,  then for all $r>1$ there exists $C$ depending on $\Omega,T, q,\gamma, \alpha, r$ such that any   solution to  \eqref{approxp} satisfies  
\begin{equation} \label{quasilinfty}
\|D u_\eps\|_{L^{r} (Q_T)} \leq C\bigg(1+ \|f_\eps\|_{L^{N+2}(Q_T)}^\frac{1}{p-1-\frac{p-2}{N+2}}+ \|Du_0\|_{L^r(\Omega)}^\frac{p-1-\frac{p-2}{N+2}}{(p-1)(N+2)-2\frac{p-2}{N+2}} \bigg)
\end{equation}
and
\begin{equation}\label{e2bis}
   \| D u_\eps \|_{L^\infty(0,T;L^r(\Omega))} \leq C\Big( 1+ \|f_\eps \|_{L^{N+2}(Q_T)}^{  {\frac{1}{(p-1)-\frac{p-2}{N+2}}}}+ \|Du_0\|_{L^r(\Omega)}\Big).
\end{equation}
Moreover for all $\omega\in(\omega_0,\infty)$
\begin{equation} \label{so2}
\||D u_\eps|^{ \omega } D u_\eps \|_{ L^{2} (0,T; H^1 (\Omega))}  
 \leq   C\Big(1+ \|f_\eps \|_{L^{N+2}(Q_T)}^\frac{2(\omega+1)}{N(N(p-1)+p)}+ \|Du_0\|_{L^\rho(\Omega)}^{\frac{\rho}{2}} \Big), 
\end{equation}
with $\omega_0 = \max\{p-2, \frac{p-2}{2}\}$.
\end{thm}

A comparison of our results with the current (parabolic) literature is in order. 

\begin{rem}\label{comparison}
The first estimate in Theorem \ref{mainintro1} extends (also with a different proof) the estimate in Theorem 4.3-(ii) in \cite{CianchiMazyaJEMS} to the time-dependent setting. Note that in the parabolic framework the power exponent of the term $\|f\|_{L^m}$ appears as the counterpart of the elliptic exponent $p-1$ found in \cite{CianchiMazyaJEMS}. We do not know if such an exponent is optimal. Moreover, in the regime $p>\frac{2N}{N+2}$ we have
\[
q>\frac{(N+2)p}{(N+2)p-N}.
\]
Therefore, our estimate extends to a wider regime the (maximal) gradient regularity found in Theorem 1.9 of \cite{BDGOjfa} or Theorem 2.3 in \cite{PorzioDCDS} (which allows lower-order terms in the equation) for source terms with integrability in the range
\[
1<m<\frac{(N+2)p}{(N+2)p-N}.
\]
The case of right-hand sides with summability
\[ 
\frac{(N+2)p}{(N+2)p-N}<m\leq m_p,\quad  p>\frac{2N}{N+2}
\]
is not contained neither in Theorem \ref{mainintro1} nor in Theorem 1.9 of \cite{BDGOjfa}, thus remains an open problem. Other gradient bounds for equations with lower order terms can be found in Section 7 of \cite{KMrmi}, but in the case of bounded source terms and coefficients of the nonlinearity in a suitable $L^p$ space. 
\end{rem}
 
\begin{rem}
A closely related approach was proposed in \cite{DiBFriedman} to prove interior gradient boundedness in $L^q$, $q<\infty$, and finer continuity results (without proof) of the gradient for quasilinear parabolic systems driven by the $p$-Laplacian having right-hand sides in $L^m$, with $m>\frac{Np}{p-1}$. It also includes a result for gradient dependent nonlinearities with growth $\gamma=p-1$, cf. Remark 7.4 in \cite{DiBFriedman}. Here, we   reach different thresholds both for $m$ and $\gamma$ due to a different use of the terms coming from the diffusion and by means of a suitable parabolic Sobolev inequality on mixed Lebesgue spaces.
\end{rem}

\begin{rem} 
We want to stress that Theorems \ref{mainintro1} and \ref{mainintro2} still hold true if we consider a gradient nonlinearity slightly more general than \eqref{modelH}. 
Indeed it is sufficient to consider $H$ satisfying 
\[
\big|H(x,t,\xi) -|\xi|^\gamma\big| \leq h(x,t,\xi) 
\quad 
\text{ with } \quad  
  |h(x,t,\xi) |\leq f (x,t) + |\xi|^{\frac{\gamma}2}  \quad 
\text{ and } \quad  
  f\in L^m (Q_T)\,, 
  \]
and one can repeat the proofs point by point. 
It is worth observing that the only variation concerns Lemma \ref{leem} because of the presence of an extra term, which can be easily handled with a suitable Young inequality. 
\end{rem}


Once Theorem \ref{mainintro2} is established, we can raise the integrability of the gradient up to $L^\infty$ combining the integral Bernstein argument with the Stampacchia approach \cite{Stampacchia}. We remark that the latter method is usually employed to prove a (qualitative) global $L^\infty$ zero-th order bound of solutions, see e.g. \cite{PorzioOrsina}. For the sake of presentation we restrict our attention to $p\geq 2$ and refer to Remark \ref{p<2} for more details on the subquadratic case and for more references.
\begin{thm}\label{mainintro3}
Let $p\geq2$, $\Omega $ be convex,  
\[
H_\eps (x,t, \xi)= (|\xi|^2+\eps)^{\frac{\gamma}2} + f_\eps (x,t)\qquad \text{ with } \quad  0\leq \gamma   <  \ell
\]
and 
$$
f_\eps(x,t ) \in L^{m} (Q_T)\, , \quad \mbox{  with }m>N+2.$$
If $u_0 \in W^{1, \infty} (\Omega)$, any   solution $u_\eps$ to \eqref{approxp}, there exists $C$ depending on $\Omega,T, q,\gamma, \alpha, r$ 
\[
\|Du_\eps\|_{L^\infty(Q_T)}\leq C.
\]
\end{thm}

\begin{rem}
Though we analyze here global estimates for problems equipped with Neumann boundary conditions, the techniques of the present paper could be extended to handle local regularity bounds on the line of \cite{DiBFriedman}. These type of estimates, along with the corresponding ones for stationary equations, will be addressed in a future work.\end{rem}

\begin{rem}
All the results stated above hold replacing $-\mathrm{div}(|Du|^{p-2}Du)$ with $-\mathrm{div}(|Du|^{p-2}Du)-\mathcal{L}u$, where $\mathcal{L}$ can be a linear uniformly elliptic operator (with appropriate Sobolev regularity assumptions on the coefficients) or a $1$-Laplacian operator. In this case the proof would be of perturbative nature. Still, we believe that the present analysis extends to operators in nondivergence form modeled on the normalized $p$-Laplacian.
\end{rem}

Remarkably, by means of the same technique we can obtain the counterpart of Theorems \ref{mainintro1}, \ref{mainintro2} for the elliptic problem
\begin{equation}\label{approxpell}
\begin{cases} 
\lambda u_\eps-\alpha_\eps(|Du_\eps|^2)\mathcal{A}_{\eps,u_\eps}(D^2u_\eps)=H_\eps (x, Du_\eps) \qquad & \mbox{in }  \Omega  \,,\\
\partial_\nu u_\eps= 0  & \mbox{on } \partial \Omega  \, \end{cases} 
\end{equation}
where $H_\eps$ is a smooth approximation of $H$. In the limit $\eps\to0$, we retrieve estimates for the model equation
\begin{equation}\label{ppell}
\begin{cases} 
\lambda u- \Delta_p u =H (x, Du ) \qquad & \mbox{in }  \Omega  \,,\\
\partial_\nu u = 0  & \mbox{on } \partial \Omega.  \, \end{cases}
\end{equation}
The results below complement those in \cite{CGL} and read as follows. 
\begin{thm}\label{mainapp}
Let $p>1$, $\Omega $ be convex,  
\[
H_\eps (x, \xi)= (|\xi|^2+\eps)^{\frac{\gamma}2} + f_\eps (x)\qquad \text{ with } \quad  0\leq \gamma   <  p-1
\]
and 
$$
f_\eps (x ) \in L^{m} ( \Omega)\, , \quad \mbox{  with } \quad m>m_p \,, \quad \mbox{ where  } \quad m_{p,\mathrm{ell}} := \max\left\{ 2, \frac{Np}{N(p-1)-(p-2)} \right\}\,.
$$ 
Then  there exists $C$ depending on 
$\Omega, N, p ,\gamma, \alpha, m$ such that any   solution $u_\eps$ to \eqref{approxpell} satisfies  
\begin{itemize} 
\item if $m_{p,\mathrm{ell}}<m<N$, then 
\begin{equation}\label{ell1}
   \|D u_\eps\|_{L^{q} (\Omega)} \leq C \Big(1+ \|f_\eps \|_{L^{m}( \Omega)}^\frac{1}{p-1 } \Big)\,, \qquad \mbox{ with } \qquad 
   q = \frac{N(p-1)m }{ N-m }
\end{equation}
and
\begin{equation} \label{soell1}
\||D u_\eps|^{ \omega } D u_\eps \|_{  H^1 (\Omega)}  
 \leq   C\Big(1+ \|f_\eps \|_{L^{m}(\Omega)}^\frac{(N-2) m}{2 (N -m)}\Big)  \, \qquad \mbox{ with } \qquad   \omega =  \frac{(N-2)m (p-1)     }{2(N -m)} -1
 \,;
\end{equation}
\item if $m=N$, then 
\begin{equation}\label{ell2}
\forall q \in[1,\infty) \qquad    \|D u_\eps\|_{L^{q} (\Omega)} \leq C \Big(1+ \|f_\eps \|_{L^{N}( \Omega)}^\frac{1}{p-1 } \Big)\,,  \end{equation}
and
\begin{equation} \label{soell2}
\forall \omega \in(\omega_0,\infty) \qquad  
\|\ |D u_\eps|^{ \omega } D u_\eps \|_{  H^1 (\Omega)}  
 \leq   C\Big(1+ \|f_\eps \|_{L^{N}(\Omega)}^\frac{\omega+1}{p-1}\Big) \,,\end{equation}
 with $\omega_0 = \max\{p-2, \frac{p-2}{2}\}$;
\item if $m>N $ and, in addition, $p\geq 2$, then  we have that $\|Du_\eps\|_{ L^\infty(\Omega)} \leq C$.
\end{itemize}
\end{thm}

\medskip 
Existence results both in the parabolic and in the elliptic framework are collected in Section \ref{exi}.

\section{Preliminary estimates via an integral (nonlinear) $p$-Bochner identity} 
Let $a,b\geq1$ and consider the anisotropic space
\[
\mathcal{V}^{a,b}(Q_T)=L^\infty(0,T;L^a(\Omega))\cap L^b(0,T;W^{1,b}(\Omega))\ .
\]
equipped with the norm
\[
\|u\|_{\mathcal{V}^{a,b}(Q_T)}\equiv \esssup_{t\in(0,T)}\|u(\cdot,t)\|_{L^a(\Omega)}
+\|Du\|_{L^b(Q_T)}
\ .
\]

We start recalling the following   inequality:

\begin{prop}
Assume that $\partial \Omega$ is piecewise smooth. Then there exists a constant $c$ depending on $N,p,r,s,T$ and $\Omega$ such that for any $v\in \mathcal{V}^{r,2}(Q_T)$ such that for $r\geq 1$ 
\begin{equation*} 
\| v \|_{L^s (Q_T)} \leq c \big( \| v \|_{L^{\infty} (0,T; L^{r} (\Omega)} + \| D v\|_{L^2 (Q_T)}\big)\,, \qquad \mbox{ where } \ s = 2 \frac{N+r}{N}\,.
\end{equation*}
Moreover we have that  
\begin{equation}\label{GaNir}
\| v \|_{L^s (Q_T)}^s \leq c \big(1+ \| v \|_{L^{\infty} (0,T; L^{r} (\Omega)}^{\frac{2r}{N}}   \| D v\|_{L^2 (Q_T)}^2 \big)  \,.
\end{equation}
\end{prop}
\proof 
For the first inequality see Proposition I.3.2 in \cite{Dib}, while the second one follows from Proposition I.3.1 therein, after a suitable use of Young inequality. 
\qed

\bigskip 

Let us recall a result regarding the behavior of $w_\eps=|Du_\eps|^2$ on the lateral surface of the parabolic cylinder $\partial \Omega \times (0,T)$. This will be useful to handle boundary integrals in the Bernstein approach.
\begin{lemma}  \label{signw}
Let $u_\eps(t) \in C^2 (\overline{\Omega})$ for all $t\in(0,T)$ with  $\Omega$  convex,    such that $\partial_\nu u_\eps =0$ on $\partial\Omega\times (0,T)$.  Then $\partial_\nu |Du_\eps|^2\leq0$ on $\partial\Omega\times (0,T)$.
\end{lemma} 
\begin{proof} A proof can be found in e.g.  \cite[Lemma 2.3]{PorrCCM}, see also the references therein. 
\end{proof}

We now state a result that contains a (nonlinear) time-dependent version of the Bochner identity.  This was already proved in \cite{LPcpde} for stationary problems and different versions were used to prove Liouville theorems \cite{Attouchi,FPS,VeronJFA}.
\begin{prop}\label{bochner} Let $u_\eps$ be a (smooth) solution to 
\begin{equation}\label{eqapp}
\partial_t u_\eps-\alpha_\eps(|Du_\eps|^2)\mathcal{A}_{\eps,u_\eps}(D^2u_\eps)=H (x,t, Du_\eps) \qquad   \mbox{in }  Q_T  \,,
\end{equation}
then $w_\eps=|Du_\eps|^2$ satisfies 
\begin{align*}
\partial_t w_\eps
&-\alpha_\eps(w_\eps) \Delta w_\eps 
+2\alpha_\eps(w_\eps)|D^2u_\eps|^2+ 
 2 \Big[ \frac{(\alpha_{\eps} ' (w_{\varepsilon}  ))^2}{\alpha_{\eps}  (w_{\varepsilon}  )}- \alpha_{\eps} '' (w_{\varepsilon}  )\Big]
   (Du_\eps\cdot Dw_\eps)^2
\\
&= 2  {\alpha_\eps'(w_\eps)} \ \mathcal{A}_{\eps,u_\eps}(D^2u_\eps)Du_\eps\cdot Dw_\eps
+ 2 Du_\eps  \cdot D H  (x,t,D u_\eps ) 
\\
&+ {\alpha_\eps' (w_\eps)} \ |Dw_\eps|^2
 + 2  \alpha'_\eps(w_\eps)   \ {D^2w_\eps D u_\eps D u_\eps}   \qquad \mbox{ in } Q_T\,.
\end{align*}
\end{prop}
\begin{proof}The proof is a consequence of Proposition 6.1 in \cite{LPcpde}, taking into account the presence of the time derivative. We report here the proof for reader's convenience. We observe that $u_\eps$ is smooth both in space and time, as the operator is uniformly parabolic. We have the following identities
\[
Dw_\eps=2D^2u_\eps Du_\eps\, ,\quad  \Delta w_\eps=2|D^2u_\eps|^2+2Du_\eps\cdot D\Delta u_\eps\, ,\quad  \partial_t w_\eps=2Du_\eps\cdot D\partial_t u_\eps\,,
\]
so that 
\begin{align*}
\alpha_\eps(w_\eps)\Delta w_\eps& = 2\alpha_\eps(w_\eps)|D^2u_\eps|^2+2\alpha_\eps(w_\eps)Du_\eps\cdot D(\Delta u_\eps)\\
&=2\alpha_\eps(w_\eps)|D^2u_\eps|^2+2Du_\eps\cdot D(\alpha_\eps(w_\eps)\Delta u_\eps)-2\alpha'(w_\eps)Du_\eps\cdot Dw_\eps\Delta u_\eps.
\end{align*}
We now use \eqref{eqapp} to find
\begin{align*}
2Du_\eps\cdot D(\alpha_\eps(w_\eps)\Delta u_\eps)&=2Du_\eps\cdot D\big( \partial_t u_\eps-H(x,t,Du_\eps)-2\alpha'(w_\eps)D^2u_\eps Du_\eps\cdot Du_\eps\big)\\
&=\partial_t w_\eps-2Du_\eps\cdot D H(x,t,Du_\eps)-2Du_\eps\cdot D\big(\alpha'(w_\eps)\underbrace{2D^2u_\eps Du_\eps}_{Dw_\eps}\cdot Du_\eps \big)\\
&=\partial_t w_\eps-2Du_\eps\cdot DH(x,t,Du_\eps)-2\alpha''(w_\eps)(Du_\eps\cdot Dw_\eps)^2-2\alpha'(w_\eps)D^2w_\eps Du_\eps\cdot Du_\eps\\
&-\alpha'(w_\eps)|Dw_\eps|^2\,,
\end{align*}
and thanks to 
\begin{equation}\label{idA}
\Delta u_\eps=\mathcal{A}_{\eps,u_\eps}(D^2u_\eps)-\frac{\alpha'(w_\eps)}{\alpha(w_\eps)}Dw_\eps\cdot Du_\eps\,,
\end{equation}
we have
\[
-2\alpha'(w_\eps)Du_\eps\cdot Dw_\eps\Delta u_\eps=2\frac{(\alpha'(w_\eps))^2}{\alpha(w_\eps)}(Du_\eps\cdot Dw_\eps)^2-2\alpha'(w_\eps)\mathcal{A}_{\eps,u_\eps}(D^2u_\eps)Du_\eps\cdot Dw_\eps\,;
\]
the conclusion now follows by gathering together  all the above equalities.
\end{proof}

Proposition \ref{bochner} is the cornerstone to apply the integral Bernstein method. The next result is basically obtained by testing the $p$-Bochner identity by  $w_\eps^\beta$ for some $\beta>0$,  and integrating the resulting equation. This is essentially the same procedure implemented in \cite{CianchiMazyaCPDE} by taking $\beta=0$ in the result below, or the one employed at the level of the weak formulation of \eqref{pp} in the stationary case in the papers \cite{Dib,DiBFriedman,CGell,CGL}.
\begin{lemma}\label{intid}
For any $\beta>0$  the following integral identity holds  for any $t \in[0,T)$:
\begin{multline}\label{main12}
\frac{1}{\beta+1}   \int_\Omega  (\eps+w_\eps(t))^{\beta+1}\,dx 
+2\intO |D^2u_\eps|^2 \alpha_\eps (w_\eps)(\eps + w_\eps )^{ \beta }\,dx\ dt \ 
\\
+  \intO \bigg[\alpha'_\eps(w_\eps)+  
 \beta \frac{\alpha_\eps(w_\eps) }{ w_\eps +\eps}
  \bigg]  (\eps+w_\eps)^{ \beta }| D w_\eps  |^2 \,dx\ dt \ 
+ 2 \beta  \intO  \big( Du_\eps\cdot Dw_\eps \big)^2      \alpha'_\eps (w_\eps) (\eps+w_\eps)^{ \beta -1 } \,dx\ dt \ 
 \\
 =  
\frac{1}{\beta+1}   \int_\Omega  (\eps+w_\eps(0))^{\beta+1}\,dx
+ 2\intO D u_\eps \cdot D H_\eps  (x,t,Du_\eps )  \  (\eps + w_\eps )^\beta\,dx\ dt \ 
\\   
+\int_0^T\int_{\partial\Omega}\alpha_\eps(w_\eps)(\eps + w_\eps )^\beta\partial_\nu w_\eps\,dS_x dt 
+
 2 \int_0^T \int_{\partial\Omega}(Dw_\eps\cdot Du_\eps)\alpha'_\eps (w_\eps) (\eps+w_\eps)^{ \beta}  \partial_\nu u\,dS_x  \ dt\,.
 \end{multline}
\end{lemma}

\begin{proof} We   multiply  the equation solved by $w_\eps$ found in Proposition  \ref{bochner} by $(\eps + w_\eps )^\beta$, integrate over $Q_T$  and get 
\begin{multline}\label{eqweps2}
 \intO (\eps + w_\eps )^{\beta} \partial_t w_\eps \,dx\ dt\ 
-\intO  \alpha_\eps(w_\eps)\Delta w_\eps \ (\eps + w_\eps )^\beta dx \ dt \ 
- \intO  {\alpha'_\eps(w_\eps)} \ |Dw_\eps|^2 (\eps + w_\eps )^{\beta}dx \ dt \ 
\\
+  2 \intO \Big[ \frac{(\alpha_{\eps} ' (w_{\varepsilon}  ))^2}{\alpha_{\eps}  (w_{\varepsilon}  )}- \alpha_{\eps} '' (w_{\varepsilon}  )\Big]
  (Du_\eps\cdot Dw_\eps)^2  (\eps + w_\eps )^{\beta}dx \ dt \ 
\\
+2\intO\alpha_\eps(w_\eps)|D^2u_\eps|^2(\eps + w_\eps )^\beta dx \ dt \ 
= 2 \intO  Du_\eps \cdot D \ H  (x,t, Du)  (\eps + w_\eps )^\beta dx \ dt \ 
 \\
+\intO2\alpha_\eps'(w_\eps)\mathcal{A}_{\eps,u_\eps}(D^2u_\eps)Du_\eps\cdot Dw_\eps (\eps + w_\eps )^\beta dx \ dt \ 
\\
+\  2\intO\alpha_\eps'(w_\eps)   \big(D^2w_\eps D u_\eps D u_\eps \big) (\eps + w_\eps )^\beta  dx \ dt
=\text{(I)+(II)+(III).}
\end{multline}
First, we notice that 
\begin{align*}
 \intO (\eps + w_\eps )^{\beta} \partial_t w_\eps \,dx\ dt \ 
=\frac{1}{\beta+1}   \int_\Omega  (\eps+w_\eps(t))^{\beta+1}\,dx-\frac{1}{\beta+1}   \int_\Omega  (\eps+w_\eps(0))^{\beta+1}\,dx\,, 
\end{align*}
and moreover applying an integration by parts with respect to  the space variable  we deduce that 
\begin{align*}
-\intO  &\alpha_\eps(w_\eps)\Delta w_\eps \ (\eps + w_\eps )^\beta dx \ dt \ 
- \intO  {\alpha'_\eps(w_\eps)} \ |Dw_\eps|^2 (\eps + w_\eps )^{\beta}dx \ dt \ 
\\
&=\intO D\big(\alpha_\eps(w_\eps)(\eps + w_\eps )^\beta \big) \ \cdot Dw_\eps\,  dx \ dt \  
-\int_0^T\int_{\partial\Omega}\alpha_\eps(w_\eps)(\eps + w_\eps )^\beta\partial_\nu w_\eps\,dS_x dt 
\\
&- \intO  {\alpha'_\eps(w_\eps)} \ |Dw_\eps|^2 (\eps + w_\eps )^{\beta}dx \ dt \ 
\\
&=
 \beta \intO     
 \alpha_\eps(w_\eps) (\eps + w_\eps )^{\beta-1} |Dw_\eps|^2 
  dx \ dt \ 
  -\int_0^T\int_{\partial\Omega}\alpha_\eps(w_\eps)(\eps + w_\eps )^\beta\partial_\nu w_\eps\,dS_x dt. 
\end{align*}

Furthermore,  using \eqref{idA} we conclude that 

\begin{align*}
\text{(II)}&=\intO 2\alpha_\eps'(w_\eps)\mathcal{A}_{\eps,u_\eps}(D^2u_\eps)Du_\eps\cdot Dw_\eps (\eps + w_\eps )^\beta  dx \ dt \  \\
&= 2 \intO   \alpha_\eps'(w_\eps) (\eps+w_\eps)^{ \beta }\left(\Delta u_\eps+\frac{ \alpha'_\eps(w_\eps)}{\alpha_\eps(w_\eps)} \big(Dw_\eps\cdot Du_\eps\big)\right) \ \big( Du_\eps\cdot Dw_\eps \big)  dx \ dt \ .
\end{align*}

We now focus on (III): integrating by parts we get 

\begin{align*}
 \text{(III)}&=2 \intO  \alpha'_\eps(w_\eps)  \big(D^2w_\eps Du_\eps\cdot Du_\eps \big) (w_\eps+\eps)^\beta  \ dx \ dt 
 \\
 &= 2 \intO \alpha'_\eps (w_\eps) (\eps+w_\eps)^{ \beta  } \sum_{i,j=1}^N\partial_{ij}w_\eps\partial_i u_\eps\partial_ju_\eps  \ dx \ dt 
 \\
 &=-2 \intO \left(\Delta u_\eps \ \big( Du_\eps\cdot Dw_\eps \big)
 +\frac12|Dw_\eps|^2\right)
 \alpha'_\eps (w_\eps) (\eps+w_\eps)^{ \beta  }   \ dx \ dt 
 \\
&- 2   \intO  \big( Du_\eps\cdot Dw_\eps \big)^2  \big[\beta  \alpha'_\eps (w_\eps) (\eps+w_\eps)^{ \beta -1 } +\alpha''_\eps (w_\eps) (\eps+w_\eps)^{ \beta  } \big] \ dx \ dt 
\\
&+
2 \int_0^T \int_{\partial\Omega}(Dw_\eps\cdot Du_\eps)  \alpha'_\eps (w_\eps) (\eps+w_\eps)^{ \beta  } \partial_\nu u\ dS_x  \ dt \,, 
 \end{align*}
so, after some cancellations, we get  
\begin{multline*}
 \text{(II)+(III)}  =
 2 \intO   \alpha_\eps'(w_\eps) (\eps+w_\eps)^{ \beta }\left(\Delta u_\eps+\frac{ \alpha_\eps '(w_\eps)}{\alpha_\eps(w_\eps)} \big(Dw_\eps\cdot Du_\eps\big)\right) \ \big( Du_\eps\cdot Dw_\eps \big)  dx \ dt \ 
\\
-2 \intO \left(\Delta u_\eps \ \big( Du_\eps\cdot Dw_\eps \big) 
 +\frac12|Dw_\eps|^2\right)
 \alpha'_\eps (w_\eps) (\eps+w_\eps)^{ \beta  }   \ dx \ dt 
 \\
- 2   \intO  \big( Du_\eps\cdot Dw_\eps \big)^2  [\beta  \alpha'_\eps (w_\eps) (\eps+w_\eps)^{ \beta -1 } +\alpha''_\eps (w_\eps) (\eps+w_\eps)^{ \beta  } ] \ dx \ dt 
\\
+
2 \int_0^T \int_{\partial\Omega}(Dw_\eps\cdot Du_\eps)  \alpha'_\eps (w_\eps) (\eps+w_\eps)^{ \beta  } \partial_\nu u\,dS_x  \ dt 
\\
 =
-  \intO   
 \alpha'_\eps (w_\eps) |Dw_\eps|^2  (\eps+w_\eps)^{ \beta  }   \ dx \ dt 
 \\
- 2   \intO  \big( Du_\eps\cdot Dw_\eps \big)^2  \left[\beta  \alpha'_\eps (w_\eps) (\eps+w_\eps)^{ \beta -1 } +\alpha''_\eps (w_\eps) (\eps+w_\eps)^{ \beta  } -  \frac{(\alpha_\eps' (w_\eps) )^2}{\alpha_\eps (w_\eps) }(\eps+w_\eps)^{ \beta  } \right] \ dx \ dt 
\\
+
2 \int_0^T\int_{\partial\Omega}(Dw_\eps\cdot Du_\eps)  \alpha'_\eps (w_\eps) (\eps+w_\eps)^{ \beta  } \partial_\nu u\,dS_x  \ dt. 
 \end{multline*}
 
Hence \eqref{eqweps2}  becomes

\begin{multline*}
\frac{1}{\beta+1}   \int_\Omega  (\eps+w_\eps(t))^{\beta+1}\,dx 
+  \intO \bigg[\alpha'_\eps(w_\eps)  +  
 \beta \frac{\alpha_\eps(w_\eps) }{ w_\eps +\eps}
  \bigg]  (\eps+w_\eps)^{ \beta }| D w_\eps  |^2  \ dx \ dt 
  \\
+2\intO |D^2u_\eps|^2 \alpha_\eps (w_\eps)(\eps + w_\eps )^{ \beta } \ dx \ dt 
 \\
 =  
\frac{1}{\beta+1}   \int_\Omega  (\eps+w_\eps(0))^{\beta+1}\,dx
+ 2\intO D u_\eps \cdot D H_\eps  (x,t,Du_\eps ) \ (\eps + w_\eps )^\beta \ dx \ dt 
\\   
+\int_0^T\int_{\partial\Omega}\alpha_\eps(w_\eps)(\eps + w_\eps )^\beta\partial_\nu w_\eps\,dS_x dt 
+
 2 \int_0^T \int_{\partial\Omega}(Dw_\eps\cdot Du_\eps)\alpha'_\eps (w_\eps) (\eps+w_\eps)^{ \beta -1}  \partial_\nu u\,dS_x  \ dt
  \\
  - 2   \intO  \big( Du_\eps\cdot Dw_\eps \big)^2   \left[\beta  \alpha'_\eps (w_\eps) (\eps+w_\eps)^{ \beta -1 } +\alpha''_\eps (w_\eps) (\eps+w_\eps)^{ \beta  } -  \frac{(\alpha_\eps' (w_\eps) )^2}{\alpha_\eps (w_\eps) }(\eps+w_\eps)^{ \beta  } \right]  \ dx \ dt 
  \\
- 
  2\intO \Big[ \frac{(\alpha_{\eps} ' (w_{\varepsilon}  ))^2}{\alpha_{\eps}  (w_{\varepsilon}  )}- \alpha_{\eps} '' (w_{\varepsilon}  )\Big]
  (Du_\eps\cdot Dw_\eps)^2  (\eps + w_\eps )^{\beta}\ dx \ dt \,,
 \end{multline*}
that gives \eqref{main12}.
\end{proof}

\begin{rem}\label{principe}
We denote by
 \begin{multline*}
\mathcal{I} =   \intO \bigg[\alpha'_\eps(w_\eps)+  
 \beta \frac{\alpha_\eps(w_\eps) }{ w_\eps +\eps}
  \bigg]  (\eps+w_\eps)^{ \beta }| D w_\eps  |^2   \ dx \ dt \\
+ 2 \beta  \intO  \big( Du_\eps\cdot Dw_\eps \big)^2      \alpha'_\eps (w_\eps) (\eps+w_\eps)^{ \beta -1 }  \ dx \ dt 
+2\intO |D^2u_\eps|^2 \alpha_\eps (w_\eps)(\eps + w_\eps )^{ \beta } \ dx \ dt 
\,.
 \end{multline*}
We observe that in the case $\alpha'<0$ the inequality $|D w_\eps|^2 \leq 4 |Du_\eps|^2 |D^2 u_\eps|^2 $ combined with \eqref{a'a} lead to 
 \begin{multline*}
 \mathcal{I} = 
  \intO 
(\eps+w_\eps)^{ \beta -1}  \bigg[ \bigg( \beta  \alpha_\eps (w_\eps) +      \alpha'_\eps(w_\eps) (\eps + w_\eps )  \bigg)  | D w_\eps  |^2 
+  2 \beta \alpha'_\eps (w_\eps)       \big( Du_\eps\cdot Dw_\eps \big)^2      \bigg] dx\ dt 
\\
+2\intO(\eps + w_\eps )^{ \beta }\alpha_\eps (w_\eps)   |D^2u_\eps|^2  dx\ dt 
\\
\geq 
  \intO 
 (\eps+w_\eps)^{ \beta -1}  \bigg[  \beta  \alpha_\eps (w_\eps) 
 +  2 \beta \alpha'_\eps (w_\eps)w_\eps             \bigg] | D w_\eps  |^2  dx\ dt 
\\
+2\intO (\eps + w_\eps )^{ \beta }  \big[\alpha_\eps (w_\eps) + 2 \alpha'_\eps (w_\eps) (w_\eps + \eps) \big]    |D^2u_\eps|^2 dx\ dt 
\\
\geq 
 \beta  \big(i_\alpha+1  \big)  \intO 
(\eps+w_\eps)^{ \beta -1} \alpha_\eps (w_\eps)   |Dw_\eps|^2   dx\ dt  
+
2(1+i_\alpha) \intO(\eps + w_\eps )^{ \beta }   \  \alpha_\eps (w_\eps)   |D^2u_\eps|^2 dx\ dt.
\end{multline*}
On the other hand, in the case
 $\alpha'\geq 0$ we have
$$
\mathcal{I} \geq 
\beta \intO 
(\eps+w_\eps)^{ \beta -1} \alpha_\eps (w_\eps)    | D w_\eps  |^2 dx\ dt  
+2  \intO (\eps + w_\eps )^{ \beta }   \ \alpha_\eps (w_\eps) 	 \  |D^2u_\eps|^2    dx\ dt 
\,.
$$
Consequently,  
$$
\mathcal{I} \geq \beta   { c_0} \intO 
(\eps+w_\eps)^{ \beta -1} \alpha_\eps (w_\eps)    | D w_\eps  |^2 dx\ dt 
+2c_0 \intO \  (\eps + w_\eps )^{ \beta }  \alpha_\eps (w_\eps)\  |D^2u_\eps|^2 dx\ dt 
\,, 
$$
where  $ { c_0}= \min\left\{ 1,     i_\alpha+1   \right\}$
.
\end{rem}

In view of Remark \ref{principe} and exploiting the Neumann boundary condition on $u_\eps$ and the sign condition on $\partial_\nu w_\eps$ deduced in  Lemma \ref{signw}, we obtain the following inequality,   $\forall \beta>0$ 
\begin{multline}	\label{main33}
   \int_\Omega  (\eps+w_\eps(t))^{\beta+1}\,dx 
+   2 c_0 (\beta+1)\intO (\eps + w_\eps )^{ \beta }  \alpha_\eps (w_\eps) |D^2u_\eps|^2 dx\ dt 
\\
+    { c_0}  \beta(\beta +1) \intO  
(\eps+w_\eps)^{ \beta -1} \alpha_\eps (w_\eps)    | D w_\eps  |^2 dx\ dt
 \\
\leq 
  \int_\Omega  (\eps+w_\eps(0))^{\beta+1}\,dx
+ 2 (\beta+1) \intO D u_\eps \cdot D H_\eps  (x,t,Du_\eps ) \  (\eps + w_\eps )^\beta dx\ dt \,.
 \end{multline}

\section{Preservation of Sobolev regularity in $W^{1,m}$, $m>2$}

This section serves as a guideline to explain the method of proof: we thus consider equations without first-order perturbations of the form
\begin{equation}\label{approxNeu}
\begin{cases}
\partial_t u_\eps-\alpha_\eps(|Du_\eps|^2)\mathcal{A}_{\eps,u_\eps}(D^2u_\eps)=f_\eps (x,t) \qquad & \mbox{in } Q_T,\\
\partial_\nu u_\eps =0 & \mbox{on } \partial \Omega\times(0,T),\\
u_\eps (x,0)=u_0(x)& \mbox{on } \Omega\,,
\end{cases}
\end{equation}
with $f_\eps$ smooth. 
Theorem \ref{conserv} shows the preservation of the (a priori) regularity in the Sobolev space $W^{1,m}$, $m\in(2,\infty]$, with respect to both the initial datum and the source term of the problem. To achieve this result we exploit partially the (powerful) integral term involving second derivatives from Remark \ref{principe}.

\begin{rem}
Note that the constant of the estimate does not depend on $p>1$. We recover the parabolic counterpart of \cite[Theorem 1.3 and Corollary 3.2]{PorrCCM} in the limit $p\to1$.
\end{rem}

\begin{proof}[Proof of Theorem \ref{conserv}]
We apply Lemma \ref{intid} with $H_\eps (x, t,  D u) = f_\eps(x,t)$. 
Since $p>1$, we can neglect the integral terms involving $|D^2u_\eps|$. 
Thus, 
recalling  that   $ { c_0} = c_0 (p) = \min\left\{ 1,     i_\alpha+1   \right\} $ (that vanishes as $p\to 1^+$),  
applying \eqref{main33} in a cylinder of height $\tau>0$,  we deuce that 
\begin{multline}\label{mainf}
  \int_\Omega  (\eps+w_\eps(t))^{\beta+1}\,dx 
+  2 c_0 (\beta +1)
\iint_{Q_\tau} |D^2u_\eps|^2 \alpha_\eps (w_\eps)(\eps + w_\eps )^{ \beta } dx\ dt  
\\
+ \beta 
(\beta+1)  c_0 \iint_{Q_\tau}    
    \alpha_\eps(w_\eps)  
  (\eps+w_\eps)^{ \beta-1 }| D w_\eps  |^2 dx\ dt  
 \\
 \leq   
   \int_\Omega  (\eps+w_\eps(0))^{\beta+1}\,dx
+2  (\beta+1)\iint_{Q_\tau} D f \cdot Du_\eps (\eps + w_\eps )^\beta dx\ dt  
 \\
 \leq   
   \int_\Omega  (\eps+w_\eps(0))^{\beta+1}\,dx
+ 2 (\beta+1)\iint_{Q_\tau} |D f | (\eps + w_\eps )^{\beta+\frac12} dx\ dt  
\,.
 \end{multline}
We then observe that, by the H\"older and  Young  inequalities,  
\begin{multline*}
2 (\beta+1) \iint_{Q_\tau} |D f_\eps |(\eps + w_\eps )^{\beta+\frac12} dx\ dt  
\leq  2(\beta+1) \|D f_\eps \|_{L^{2(\beta+1)}(Q_\tau)}\left(\int_0^{\tau}  \int_\Omega  (\eps+w_\eps)^{\beta+1} dx\ dt   \right)^{\frac{2\beta+1}{2(\beta+1)}}
\\
\leq 2 (\beta+1) \|D f_\eps \|_{L^{2(\beta+1)}(Q_\tau)}
\left(\tau \ \esssup_{t\in[0,\tau]}  \int_\Omega (\eps+w_\eps(t))^{\beta+1} dx \right)^{\frac{2\beta+1}{2(\beta+1)}}
\\
\leq 
\frac{1}{2 }  \ \|D f_\eps \|_{L^{2(\beta+1)}(Q_\tau)}^{2(\beta+1)}
+ \frac{2\beta+1}{2^{-\frac{1}{2 \beta +1}}} \ \tau \ \esssup_{t\in[0,\tau]}  \int_\Omega (\eps+w_\eps(t))^{\beta+1}  dx \,.
\end{multline*}
Dropping   two positive terms on the left hand side in \eqref{mainf} and taking the supremum with respect to $t\in [0,\tau]$, choosing  $ \tau\leq\frac{1}{(2\beta+1) \ 2^{\frac{2\beta}{2\beta+1}}}  $ we end up with
$$\frac12 
\esssup_{t\in[0,\tau]}  \int_\Omega  (\eps +w_\eps(t))^{\beta+1}  dx \  
\leq \frac{ 1  }2 \iint_{Q_\tau}|D f_\eps |^{2(\beta+1)}  dx\ dt  
+    \int_\Omega  (\eps + w_\eps(0))^{\beta+1} dx \,, 
$$
that yields to 
$$
\left\|   w_\eps (x,t) \right\|_{L^{\infty} (0,\tau; L^{ \beta+1 } (\Omega))}^{ \beta+1 } 
\leq    \|D f_\eps \|_{L^{2(\beta+1)} (Q_\tau)}^{2(\beta+1)} 
+ 
 2 \left\| \eps + w_\eps (0) \right\|_{  L^{\beta+1 } (\Omega)}^{ \beta+1 }  \,.
$$
Iterating the estimates $k+1$ times, where $k$ is the integer part of $T/\tau$, we conclude that
$$
\left\| D u_\eps (x,t) \right\|_{L^{\infty} (0,T; L^{2 (\beta+1)} (\Omega))}^{2(\beta+1)} 
\leq    C(\beta,T)\Big(\|D f_\eps \|_{L^{2(\beta+1)} (Q_T)}^{2(\beta+1)} 
+ 
  \left\|\eps + |D u_\eps (0)|^2 \right\|_{  L^{ \beta+1 } (\Omega)}^{ \beta+1 } \Big)  \,.
$$
where $C(\beta,t)\to \overline{C}$ as $\beta\to\infty$.
\end{proof}

\begin{rem}
The previous argument does not allow us to obtain the conservation of the regularity when $1<m<2$ (neither for $1<p<2$ nor for $p>2$) and this remains at this stage an open problem using the integral Bernstein technique.
\end{rem}

\section{A priori estimates for equations with right-hand side in $L^m$}

In order to estimate the gradient of the solutions to \eqref{pp}, let us  specialize \eqref{main33}  in the case of Hamiltonians that are sum of a gradient dependent term with power-growth and a space-time source $f(x,t)$.  
\begin{lemma}  \label{leem}
Let $H_\eps (x,t, \xi) = (|\xi|^2+\eps)^{\frac{\gamma}2} + f_\eps(x,t) $ 
then 
\begin{multline} \label{lem}
\left| \intO D u_\eps \cdot D H_\eps  (x,t,Du_\eps ) \  (\eps + w_\eps )^\beta dx\ dt  \right|
\\
\leq     \frac{ c_0}2
\intO     \  (\eps + w_\eps )^{ \beta } \alpha_\eps (w_\eps) |D^2  u_\eps |^2    dx\ dt  
+\frac{\beta c_0 }{4} \intO    \ (\eps + w_\eps )^{\beta }  \alpha_\eps (w_\eps) \    |D w_\eps |^2  dx\ dt 
\\
+   \frac{\gamma^2}{(\gamma +2\beta)^2 }   \frac{N}{  c_0}  \intO        \  \frac{(\eps + w_\eps )^{ {\gamma}   +\beta } }{\alpha_\eps (w_\eps)}dx\ dt  
+\frac{ N+\beta }{  c_0} \intO |f_\eps (x,t) |^2  \frac{(\eps + w_\eps )^\beta }{ \alpha_\eps (w_\eps) } dx\ dt   \,.
\end{multline}
\end{lemma}

\begin{proof}
By the chain rule and integrating by parts two times we have 
\begin{multline*} 
  \intO  Du_\eps \cdot D \big( (\eps + w_\eps )^{\frac{\gamma}2} + f_\eps (x,t) \big)(\eps + w_\eps )^\beta  dx\ dt  
 = 
\frac{\gamma}2   \intO  Du_\eps \cdot D w_\eps \  (\eps + w_\eps )^{\frac{\gamma}2-1 +\beta } dx\ dt  
\\ 
   -  \intO f_\eps (x,t)  \text{ div } \Big(Du_\eps (\eps + w_\eps )^\beta \Big)  dx\ dt 
+      \int_{\partial \Omega \times (0,T)}  
\underbrace{f_\eps (x,t)  (\eps + w_\eps )^\beta  \Big( D u_\eps  \cdot \nu  \Big) dS_x \ dt }_{=0} 
\\
= 
- \frac{\gamma}{\gamma +2\beta }    \intO  \Delta u_\eps       \  (\eps + w_\eps )^{\frac{\gamma}2  +\beta } dx\ dt  
-    \intO f_\eps (x,t)  \big[ \Delta u_\eps (\eps + w_\eps )^\beta + \beta D u_\eps \cdot D w_\eps \ (\eps + w_\eps )^{\beta-1} \Big]   dx\ dt.
 \end{multline*}

Recalling that by the Cauchy-Schwarz inequality $|\Delta u_\eps | \leq \sqrt{N} |D^2 u_\eps|$ we have 

$$
\left|\intO D u_\eps \cdot D H_\eps  (x,t,Du_\eps ) \  (\eps + w_\eps )^\beta dx\ dt   \right|
$$
$$
\leq   \frac{\gamma}{\gamma +2\beta }   \sqrt{N}  \intO  |D^2  u_\eps |      \  (\eps + w_\eps )^{\frac{\gamma}2  +\beta } dx\ dt  
+
    \intO |f_\eps (x,t) | \big[\sqrt{N}  |D^2 u_\eps|  + \beta  | D w_\eps | \Big] (\eps + w_\eps )^\beta  dx\ dt 
$$
$$
\leq   \frac{c_0}{ 4  }   
\intO     \  (\eps + w_\eps )^{ \beta } \alpha_\eps (w_\eps) |D^2  u_\eps |^2    dx\ dt  
+   \frac{\gamma^2}{(\gamma +2\beta)^2 }   \frac{N}{  c_0}  \intO        \  \frac{(\eps + w_\eps )^{ {\gamma}   +\beta } }{\alpha_\eps (w_\eps)}dx\ dt  
$$ 
$$   
+ \frac{c_0}4 \intO     (\eps + w_\eps )^\beta  \alpha_\eps (w_\eps) |D^2 u_\eps|^2 dx\ dt 
+ \frac{ N}{ c_0} \intO |f_\eps (x,t) |^2  \frac{(\eps + w_\eps )^\beta }{ \alpha_\eps (w_\eps) } dx\ dt 
$$ $$
+\frac{\beta c_0 }4 \intO    \ (\eps + w_\eps )^{\beta }  \alpha_\eps (w_\eps) \    |D w_\eps |^2  dx\ dt 
+ \frac{ \beta}{  c_0} \intO  |f_\eps (x,t) |^2      \frac{(\eps + w_\eps )^{\beta } }{\alpha_\eps (w_\eps)}   dx\ dt \,,
$$
and thus \eqref{lem} holds.
\end{proof}

Plugging \eqref{lem} into \eqref{main33} we finally get the following inequality that is the milestone of our study. 

 \begin{lemma} \label{start}
 Let $u_\eps$ be a $C^{2} (Q_T)$ solution to \eqref{approxp}, then $w_\eps = |D u_\eps|^2$ satisfies,  $\forall \beta\geq1$ 

\begin{multline}\label{main3}
   \int_\Omega  (\eps+w_\eps(t))^{\beta+1}\,dx 
+  c_0 (\beta+1)  \intO (\eps + w_\eps )^{ \beta }  \alpha_\eps (w_\eps) |D^2u_\eps|^2 dx\ dt 
\\
+    { c_0} \frac{ \beta(\beta +1)}2 \intO  
(\eps+w_\eps)^{ \beta -1} \alpha_\eps (w_\eps)    | D w_\eps  |^2 dx\ dt
\leq 
  \int_\Omega  (\eps+w_\eps(0))^{\beta+1}\,dx
   \\
+ c_{1,\beta}   \intO |f_\eps (x,t) |^2  \frac{(\eps + w_\eps )^\beta }{ \alpha_\eps (w_\eps) } dx\ dt 
+  c_{2,\beta}       \intO        \  \frac{(\eps + w_\eps )^{ {\gamma}   +\beta } }{\alpha_\eps (w_\eps)} dx\ dt  
  \,,
  \end{multline}
where 
$$
 { c_0}= \min\left\{ 1,     i_\alpha+1   \right\}\,, \qquad   c_{1, \beta}   = 2 \frac{ (\beta+1) \big ( {N}  + \beta  \big)}{ c_0} 
  \qquad \mbox{ and } \qquad
  c_{2, \beta}  = \frac{N}{ c_0}  \frac{\gamma^2  (\beta+1) }{(\gamma +2\beta)^2 }  \,.
$$
\end{lemma}

We are now ready to provide the proof of the gradient estimates for equations with source terms in $L^m$. 
\begin{proof}[Proof of Theorem \ref{mainintro1}]
We set in \eqref{main3}
$$
\beta= 2 	\mu - \frac{p}2 \qquad  \text{with } \qquad \mu=\frac{Nm (p-1)-(p-2) (m-2)}{4 (N+2-m)}\,.
$$
Since $m>m_p$, this choice ensures that
\[
\beta>0\ \left(\text{i.e. }\mu>\frac{p}{4}\right)\,, 
\]
and hence we are allowed to apply Lemma \ref{intid}. We have, using \eqref{Aa},
\begin{multline}\label{main51}
   \int_\Omega  (\eps+w_\eps(t))^{2 	\mu - \frac{p}2+1}\,dx 
+ {c_0}  \big(2 	\mu - \frac{p}2+1\big)  \intO (\eps + w_\eps )^{ 2 	\mu - \frac{p}2 } 	\ \alpha_\eps (w_\eps) \ |D^2u_\eps|^2  \ dx dt \ 
\\
+ \frac{\underline{\alpha} (2 	\mu - \frac{p}2)  \big(2 	\mu - \frac{p}2+1\big)   { c_0} }2 
\intO 
(\eps+w_\eps)^{ 2 	\mu -2}    | D w_\eps  |^2\ dx dt \ 
 \\
\leq 
  \int_\Omega  (\eps+w_\eps(0))^{2 	\mu - \frac{p}2+1}\,dx
+  
\frac{ \hat{c}_{1 , \mu}}{\underline{\alpha} }  \intO |f (x,t ) |^2     | (\eps + w_\eps )^{2\mu -p   +1 }   dx \ dt  
+  
\frac{\hat{c}_{2 , \mu}}{\underline{\alpha} }   \intO        (\eps + w_\eps )^{\gamma+2\mu -p   +1 }  dx \ dt  \,,
 \end{multline}
where 
\begin{equation}\label{ci} 
 \hat{c}_{1 , \mu} =  c_{1, 2 	\mu - \frac{p}2 }   = \frac{ 2}{ c_0} \Big(2 	\mu - \frac{p-2}2   \Big) \Big(2 	\mu+ \frac{2N-p}2      \Big) \quad 
\mbox{ and }
 \quad \hat{c}_{2 , \ \mu} =  c_{2, 2 	\mu - \frac{p}2 }   = 
 \frac{N}{  c_0}  \frac{\gamma^2  (2\mu -\frac{p}2 +1) }{(\gamma +4 \mu - p +2 )^2}\,.
\end{equation} 
We write
\[
\intO (\eps+w_\eps)^{ 2 	\mu -2}    | D w_\eps  |^2\ dx dt=\frac{1}{\mu^2}\intO |D(\eps+w_\eps)^\mu|^2\ dx dt,
\]
use that
\[
{c_0}  \big(2 	\mu - \frac{p}2+1\big)  \intO (\eps + w_\eps )^{ 2 	\mu - \frac{p}2 } 	\ \alpha_\eps (w_\eps) \ |D^2u_\eps|^2  \ dx dt\geq0
\]
and get from \eqref{main51} the inequality  
 \begin{equation*}
\begin{array}{c}\dys 
      \int_\Omega  (\eps+w_\eps(t))^{2 	\mu - \frac{p}2+1}\,dx
+ 
\frac{\underline{\alpha} (2 	\mu - \frac{p}2)  \big(2 	\mu - \frac{p}2+1\big)   { c_0} }{ 2  \mu^2}
\intO |D(\eps+w_\eps)^\mu|^2\ dx\  dt \ 
\\ \dys 
\leq 
 \int_\Omega  (\eps+w_\eps(0))^{2 	\mu - \frac{p}2+1}\,dx
+  
\frac{ \hat{c}_{1 , \mu}}{\underline{\alpha} }  \intO |f (x,t ) |^2     | (\eps + w_\eps )^{2\mu -p   +1 }   dx \ dt  
+  
\frac{\hat{c}_{2 , \mu}}{\underline{\alpha} }   \intO        (\eps + w_\eps )^{\gamma+2\mu -p   +1 }  dx \ dt\,.\end{array}
 \end{equation*}
Thanks to H\"older inequality with exponents $m/2$ and $m/(m-2)$ we have that   
 \begin{equation*} 
\begin{array}{c}\dys 
     \int_\Omega  [(\eps+w_\eps(t))^\mu]^{2 - \frac{p-2}{2\mu}}\,dx
+ 
\frac{\underline{\alpha} (2 	\mu - \frac{p}2)  \big(2 	\mu - \frac{p}2+1\big)   { c_0} }{ 2  \mu^2}
\intO |D(\eps+w_\eps)^\mu|^2\ dx dt \ 
\leq 
   \int_\Omega  (\eps+w_\eps(0))^{2 	\mu - \frac{p}2+1}\,dx
\\ \dys 
+  \frac{ \hat{c}_{1, \mu}}{\underline{\alpha} }    \| f_\eps (x,t) \|^2_{L^m (Q_T)}   \| (\eps+w_\eps)^\mu \|^{2- \frac{(p-1)}{\mu}}_{L^{\frac{2\mu-(p-1)}{\mu}\frac{m}{m-2}}(Q_T) }   \ 
+  \frac{ \hat{c}_{2, \mu}}{\underline{\alpha} } \intO        (\eps + w_\eps )^{\gamma+2\mu -p   +1 }  dx \ dt \,, 
 \end{array}
 \end{equation*}
 We are now in position to use the embedding in \eqref{GaNir}: we apply it with $v=(\eps+w_\eps)^\mu$, $r=2  - \frac{p-2}{2\mu}>1$ 
 and $s=2\frac{N+r}{N}=2\frac{N+2}N  - \frac{p-2}{ \mu N}$. This leads to 
\begin{equation*} 
\begin{array}{c}\dys 
\frac12      \int_\Omega  (\eps+w_\eps(t))^{2 	\mu - \frac{p}2+1}\,dx 
+
c_1 \bigg( \intO (\eps+w_\eps)^{\mu s} \ dx \ dt \ 
\bigg)^{\frac{N}{N+2}}
\leq c \bigg( 1+ 
\int_\Omega  (\eps+w_\eps(0))^{2 \mu - \frac{p}2+1}\,dx
\\ \dys 
+     \| f_\eps (x,t) \|^2_{L^m (Q_T)}   \|  \eps+w_\eps  \|^{2\mu -  {(p-1)}}_{L^{\frac{2\mu-(p-1)}{\mu}\frac{m}{m-2}}(Q_T) }  \ 
+    \|   \eps+ w_\eps\|^{2 \mu + \gamma - (p- 1) }_{L^{2 \mu + \gamma - (p- 1)  } (Q_T) }  \bigg)  \,,
 \end{array}
 \end{equation*}
for a constant $c_1$ such that 
$$
0<c_1 = c_1 (p,\alpha,\mu) \leq  \min\bigg\{ \frac{\underline{\alpha} (2 	\mu - \frac{p}2)  \big(2 	\mu - \frac{p}2+1\big)   { c_0} }{ 2  \mu^2}, \frac 12 \bigg\}\,.
$$
By passing to the supremum with respect to the  time variable on the left-hand side we conclude  
 \begin{equation}\label{main556}
\begin{array}{c}\dys 
    \|   \eps+ w_\eps\|^{ 2\mu -\frac{p- 2}{ 2   }}_{L^{\infty} (0,T; L^{2\mu -\frac{p- 2}{ 2   }}(\Omega))}  
+
    \| \eps+ w_\eps \|^{\frac{Ns \mu }{N+2}}_{L^{\mu s} (Q_T)}
\\ \dys 
\leq 
c \bigg(
  1+  \| \eps+ w_\eps (0)\|^{ 2\mu-\frac{p- 2}{ 2   }}_{L^{ 2\mu-\frac{p- 2}{ 2   }} (\Omega)}
+      \| f_\eps (x,t) \|^2_{L^m (Q_T)}   \|  \eps+ w_\eps\|^{2\mu- ( p-1)}_{L^{  \frac{m[2\mu- ( p-1)]}{m-2}} (Q_T) }   \ 
+   \|          \eps+ w_\eps \|^{2 \mu + \gamma - (p- 1) }_{L^{2 \mu + \gamma - (p- 1)  } (Q_T) }   \bigg) \,.
 \end{array}
 \end{equation}
In view of the choice of $\mu$, we have that 
\[
\frac{m[2\mu- ( p-1)]}{m-2}= \mu s\,, 
\]
and  $2\mu-(p-1)>0$ since $p>1$. Moreover 
 $$
2\mu- ( p-1)<  \frac{Ns \mu }{N+2}  = 2\mu - \frac{p-2}{N+2}
 $$
and consequently  \eqref{main556} can be rephrased as 
  \begin{equation*} 
\begin{array}{c}\dys 
    \|  \eps+ w_\eps \|^{ 2\mu -\frac{p- 2}{ 2   }}_{L^{\infty} (0,T; L^{2\mu -\frac{p- 2}{ 2   }} (\Omega))}  
+
    \| \eps+ w_\eps \|^{\frac{Ns \mu }{N+2}}_{L^{\mu s} (Q_T)}
\\ \dys 
\leq 
c \bigg(
  1+  \| \eps+ w_\eps (0)\|^{ 2\mu-\frac{p- 2}{ 2   }}_{L^{ 2\mu-\frac{p- 2}{ 2   }} (\Omega)}
+      \| f_\eps (x,t) \|^2_{L^m (Q_T)}   \|  \eps+ w_\eps\|^{2\mu- ( p-1)}_{L^{  \mu s } (Q_T) }   \ 
+   \|          \eps+ w_\eps\|^{2 \mu + \gamma - (p- 1) }_{L^{2 \mu + \gamma - (p- 1)  } (Q_T) }   \bigg).
 \end{array}
 \end{equation*}
We now apply   Young inequality with the duality pair 
\[
\nu= \frac{\frac{Ns \mu }{N+2}}{2\mu -(p-1)}= \frac{2\mu -\frac{p-2}{N+2}}{2\mu -(p-1)} = \frac{Nm}{(N+2)(m-2)},
\]
\[
  \nu'=\frac{\frac{Ns \mu }{N+2}}{p-1- \frac{p-2}{N+2}}=\frac{2\mu -\frac{p-2}{N+2}}{ p-1- \frac{p-2}{N+2}}  = \frac12 \frac{Nm}{N+2-m},
  \]
 so that 
   \begin{equation*} 
\begin{array}{c}\dys 
    \|  \eps + w_\eps\|^{ 2\mu -\frac{p- 2}{ 2   }}_{L^{\infty} (0,T; L^{2\mu -\frac{p- 2}{ 2   }} (\Omega))}  
+
    \|\eps +w_\eps \|^{\frac{Ns \mu }{N+2}}_{L^{\mu s} (Q_T)}
\\ \dys 
\leq 
c \bigg(
  1+  \|\eps +w_\eps (0)\|^{ 2\mu-\frac{p- 2}{ 2   }}_{L^{ 2\mu-\frac{p- 2}{ 2   }} (\Omega)}
+      \| f_\eps (x,t) \|_{L^m (Q_T)}^{2\nu'}       \ 
+   \|   \eps +      w_\eps\|^{2 \mu + \gamma - (p- 1) }_{L^{2 \mu + \gamma - (p- 1)  } (Q_T) }   \bigg) \,,
 \end{array}
 \end{equation*}
and 
 \[
 2\mu-\frac{p-2}{2}=\frac{N[m(p-1)-(p-2)]}{2(N+2-m)}.
 \]

 \medskip 
 
Suppose now that  $0<\gamma <   \frac{p}2$, then 
$$
2 \mu + \gamma -    (p- 1) < 2 \mu -  \frac{   {p}-2  }{2 }
$$
so that by the weighted Young inequality we have
$$
         c\intO ( \eps+ w_\eps)^{2 \mu + \gamma - (p- 1)  }
         \leq \frac{1}{2}\intO ( \eps+ w_\eps)^{2 \mu -  \frac{   {p}-2  }{2 }}  + c'.
$$
Then
   \begin{equation}\label{main5c1}
\begin{array}{c}\dys 
    \|  \eps+ w_\eps\|^{ 2\mu -\frac{p- 2}{ 2   }}_{L^{\infty} (0,T; L^{2\mu -\frac{p- 2}{ 2   }} (\Omega))}  
+
    \| \eps+ w_\eps \|^{\frac{Ns \mu }{N+2}}_{L^{\mu s} (Q_T)}
\leq 
c \bigg(
  1+  \|\eps+ w_\eps (0)\|^{ 2\mu-\frac{p- 2}{ 2   }}_{L^{ 2\mu-\frac{p- 2}{ 2   }} (\Omega)}
+      \| f_\eps (x,t) \|_{L^m (Q_T)}^{2\frac{\frac{Ns \mu }{N+2}}{ p-1- \frac{p-2}{N+2}}}       \ 
   \bigg)
 \end{array}
 \end{equation}
 holds true.

On the other hand, if $0<\gamma  <p-1- \frac{p-2}{N+2} $, 
then 
$$
2 \mu + \gamma -    (p- 1) < \frac{N s \mu}{N+2} = 2 \mu - \frac{p-2}{N+2} < s \mu\,,
$$
so that  
\[   \intO     (\eps+ w_\eps)^{2 \mu + \gamma - (p- 1)}  \ dx \ dt 
\leq \frac1{2c}     \| \eps+ w_\eps \|^{\frac{N s \mu}{N+2} }_{L^{ \mu s    } (Q_T) }
+ c' \,,
\]
and thus \eqref{main5c1} holds true,  so that  \eqref{e1}    follows. 
As far as \eqref{e2} is concerned, it is directly deduced by \eqref{main5c1}. 

\medskip 

Finally,  \eqref{so} can be deduced  from \eqref{main51} repeating the above proof without dropping the second order term. 
\end{proof}

\bigskip

\begin{proof}[Proof of Theorem \ref{mainintro2}]
The proof follows from Theorem \ref{mainintro1} by letting $m\nearrow N+2$.
\end{proof}

 \begin{rem}
We observe that Theorem \ref{mainintro2} holds up to the end-point case  $\gamma =  \ell = \max\left\{\frac{p}2,p-1-\frac{p-2}{N+2} \right\}$.
Indeed, the main modification concerns the right-hand side of \eqref{main51}: since it has the same growth of the terms appearing on the left-hand side, by tracking back all the constants it would follow (see \eqref{ci}) that the constant $\hat{C_2}$ vanishes as $\mu $ diverges. This allows to reabsorb it on the left-hand side of the inequality. 
\end{rem}

\begin{rem}
The above proof provides also the following integral estimate
\[
\iint_{Q_T}|Du|^{2(2\mu-1)}|D^2u|^2\,dxdt\lesssim C(\|f\|_m).
\] 
These kind of estimates have been recently studied by different methods in \cite{Montoro} for stationary equations.
\end{rem}

\begin{proof}[Proof of Theorem \ref{mainintro3}]
We start observing that, by Theorem \ref{mainintro2}, we have that   $\widetilde{f}(x,t):= (|Du_\eps|^2+\eps)^\gamma+f(x,t)$ belongs to $L^m$, $m>N+2$. 
We can thus consider an evolutionary $p$-Laplacian type equation with right-hand side $\widetilde{f}(x,t)$ having summability above the critical dimension $N+2$. We take
\[
\beta=2\mu-\frac{p}{2}>0
\]
and slightly modify the proof of Lemma \ref{bochner} by using the test function $(\eps+w_\eps-k)_+^\beta$. We emphasize that $\mu$ only needs to satisfy $\mu>\frac{p}{4}$ and does not play any role in the course of the proof as in Theorem \ref{mainintro1}. Owing to Remark \ref{principe} we conclude
\begin{multline*} 
\frac{1}{2\mu-\frac{p}{2}+1}   \int_\Omega  (\eps+w_\eps (t)  - k)_+^{2\mu-\frac{p}{2}+1}\,dx 
 \\
+ 
 c_0 (2\mu-\frac{p}{2})   \iint_{Q_T}    \alpha_\eps(w_\eps)(\eps+w_\eps - k)_+^{2\mu-\frac{p}{2}-1}
    | D w_\eps  |^2 \,dxdt
+2 c_0 \iint_{Q_T} |D^2u_\eps|^2 \alpha_\eps (w_\eps)  (\eps+w_\eps - k)_+^{ 2\mu-\frac{p}{2} }  \,dxdt
\\
\leq   \frac{1}{2\mu-\frac{p}{2}+1}   \int_\Omega  (\eps+w_\eps (0)- k)_+^{2\mu-\frac{p}{2}+1}\,dx
+ 2\iint_{Q_T} D u_\eps \cdot D \widetilde{f}_\eps  (x,t  ) \ (\eps+w_\eps - k)_+^{2\mu-\frac{p}{2}}\,dxdt.
    \end{multline*}
Using that $p\geq2$,  and taking $k\geq \|w_\eps(0)\|_{L^\infty(\Omega)}+\eps$ to neglect the term involving the initial datum, we get by integrating by parts the term involving $\widetilde{f}$
\begin{multline*}
\frac{1}{2\mu-\frac{p}{2}+1}   \int_\Omega  (\eps+w_\eps (t)  - k)_+^{2\mu-\frac{p}{2}+1}\,dx 
+ 
\frac12  c_0 (2\mu-\frac{p}{2})   \iint_{Q_T}   (\eps+w_\eps - k)_+^{2\mu -2}
    | D w_\eps  |^2 \,dxdt
\\
\leq    
 2 {N}\iint_{Q_T}|\tilde{f}_\eps  (x,t  )|^2     (\eps+w_\eps - k)_+^{2\mu-(p-1)}\,dxdt+\frac{1}{2 c_0 (2\mu-\frac{p}{2})}\iint_{Q_T}\underbrace{|\tilde{f}_\eps  (x,t  )|^2 \  |D u_\eps |^2 }_{\tilde{g}^2_\eps (x,t, w_\eps)}  \frac{ (\eps+w_\eps - k)_+^{2\mu-\frac{p}{2}-1}}{\alpha_\eps (w_\eps)}\,dxdt,
    \end{multline*}
that means 
\begin{multline*}\label{main3.4}
\frac{1}{2\mu-\frac{p}{2}+1}   \int_\Omega  [(\eps+w_\eps (t)  - k)_+^\mu]^{ 2-\frac{p-2}{2\mu}}\,dx 
+ 
\frac12  \frac{c_0 (2\mu-\frac{p}{2}) }{\mu^2}  \iint_{Q_T}   [D(\eps+w_\eps - k)_+^\mu ]^2 \,dxdt
\\
\leq    
2N\iint_{Q_T}|\tilde{f}_\eps  (x,t  )|^2     [(\eps+w_\eps - k)_+^\mu]^{2-\frac{p-1}{\mu}}\,dxdt+\frac{1}{2 c_0 (2\mu-\frac{p}{2})}\iint_{Q_T}|\tilde{g}_\eps (x,t, w_\eps)|^2  [(\eps+w_\eps - k)_+^{\mu}]^{2-\frac{p}{\mu}}\,dxdt.
    \end{multline*}
We denote by $A_{k,\eps}:=\{(x,t)\in Q_T: \ \eps + w_\eps(x,t)\geq k\}$. We take care only about the first integral of the right-hand side, the second one having a slower growth with respect to $w_\eps$ (note that also $\tilde g\in L^m$, $m>N+2$, since $Du\in L^r$, for any $r<\infty$). Applying the Gagliardo-Nirenberg inequality on the left-hand side, and the H\"older inequality on the right-hand side, we end up with
$$    \|(\eps+w_\eps (t)  - k)_+^{\mu} \|_{L^{s} (A_{k,\eps})}^{\frac{Ns}{N+2}}  
\leq    
c_1\|\tilde{h}_\eps  (x,t  )\|_{L^{2\eta} (A_{k,\eps})}^2  \| (\eps+w_\eps - k)_+^{\mu}\|_{L^{s} (A_{k,\eps})}^{2-\frac{p-1}{\mu}},     $$
where \[
s= 2 \frac{N+2}{N}- \frac{p-2}{\mu N}\,, \quad 
\eta= \frac{s}{s-2 +\frac{p-1}{\mu}} = \left(\frac{s}{2-\frac{p-1}{\mu}}\right)' 
\qquad \mbox{ and  } \quad |\tilde{h}_\eps  (x,t  )|^2= |\tilde{f}_\eps (x,t, w_\eps)|^2  + |\tilde{g}_\eps (x,t, w_\eps)|^2 \,.
\]
A further application of the H\"older inequality leads to
$$    \|(\eps+w_\eps (t)  - k)_+^{\mu} \|_{L^{s} (A_{k,\eps})}^{\frac{Ns}{N+2}-2+\frac{p-1}{\mu} }  
\leq    
c_2\|\tilde{f}_\eps  (x,t  )\|_{L^{m} (A_{k,\eps})}^2  |A_k|^{\frac{1}{\eta}-\frac2{m}}.
    $$
As a consequence, for $h>k$ we have $A_{h,\eps}\subset A_{k,\eps}$ (thus the function $\sigma\longmapsto |A_{\sigma,\eps}|$ is nonincreasing) and
$$ \|(\eps+w_\eps (t)  - k)_+^{\mu} \|_{L^{s} (A_{k,\eps})}
\geq 
(h-k)^\mu |A_{h,\eps}|^{\frac1s}\,,$$
that  implies the following inequality
$$  
[(h-k)^\mu |A_{h,\eps}|^{\frac1s}]^{\frac{Ns}{N+2}-2+\frac{p-1}{\mu} }  
\leq    
c_2\|\tilde{h}_\eps  (x,t  )\|_{L^{m} (A_{k,\eps})}^2  |A_{k,\eps}|^{\frac{1}{\eta}-\frac2{m}}\,,
    $$
and hence
$$  
 |A_{h,\eps}| \leq    
c_3\|\tilde{h}_\eps  (x,t  )\|_{L^{m} (A_{k,\eps})}^2  
\frac{ |A_{k,\eps}|^{s\frac{\frac{1}{\eta}-\frac2{m}}{{\frac{Ns}{N+2}-2+\frac{p-1}{\mu} }  
} }}
{ (h-k)^{\mu ({\frac{Ns}{N+2}-2+\frac{p-1}{\mu} )}  }
}.
  $$
We are now in position to apply Lemma 5.1 in 
\cite{Stampacchia}, to the measure of the superlevel sets, whenever
 $$
 \mu \left(\frac{Ns}{N+2}-2+\frac{p-1}{\mu} \right)   >0
 $$
and 
$$
s\frac{\frac{1}{\eta}-\frac2{m}}{{\frac{Ns}{N+2}-2+\frac{p-1}{\mu}}} >1
\qquad 
\iff
\qquad 
{\frac{1}{\eta}-\frac2{m}} 
> \frac{N }{N+2}-\frac2{s}+\frac{p-1}{s\mu}.
$$
Recalling the choices of $s$ and $\eta$ 
the above inequality turns into 
$$
\frac{2 }{N+2} =1 - \frac{N }{N+2} 
> \frac2{m}.
$$
Therefore we can conclude that for  $m>N+2$ there exists $k_0$ such that $|A_{k,\eps}|\equiv 0$ for any $k\geq k_0$, which implies $\esssup_{Q_T} w_\eps\leq k_0$.
\end{proof}

 \begin{rem}\label{p<2}
The proof of Theorem \ref{mainintro3} combining the Bernstein method and the Stampacchia approach is new, although it is restricted to $p\geq2$. It appeared (in a local form) using a De Giorgi-Moser iteration in several papers \cite{DiBFriedman,Choe} and it covers the whole range $p>1$, see also Lemma VIII.4.1 and VIII.4.2 of \cite{Dib}. The authors in \cite{DiBFriedman} pointed out the validity of the $L^\infty$ bound when $m>Np'$ (also under Neumann boundary conditions, without convexity assumptions on the domain), while \cite{Dib} considered $m>N+2$. Both papers treated $\gamma\leq p-1$, but all the results were stated without proof, including those for Neumann boundary conditions, cf. Remark 7.4 and Theorem 7.4 in \cite{DiBFriedman}.  The paper \cite{Misawa} gave the proof of an interior $L^\infty$ bound using similar methods, but under a different growth condition for $H$ depending on $p$ and $N$. The description of the parabolic sublinear region defined in the introduction of this manuscript was explained in detail in e.g. \cite{Magliocca}.
\end{rem}


\begin{rem}[on the convexity assumption] 
Let us stress that the hypothesis on the convexity of the domain can be removed in some cases. As it can observed in the proof of the main results, such assumption  is crucial since, together with the Neumann boundary condition, it guarantees   that $\partial_\nu w_\eps <0$ at $\partial \Omega$.  
To avoid such a requirement, one might try to write the equivalent of inequality \eqref{main33} for $z_\eps = w_\eps e^{\lambda d (x)}$
where $d(x)=\mathrm{dist}(x, \partial \Omega)$ and $\lambda$ is a positive constant depending on the main curvature of $\partial \Omega$.   Indeed in this case we have $\partial_\nu z_\eps <0$ at $\partial \Omega$ (see for instance \cite{LPSiam}, Lemma 2.4). 
These changes lead to the following inequality 
\begin{multline*} 
   \int_\Omega  (\eps+z_\eps(t))^{\beta+1}\,dx 
+     c_0 (\beta+1)\intO (\eps + z_\eps )^{ \beta }  \alpha_\eps (w_\eps) |D^2u_\eps|^2 dx\ dt 
+   \frac12 { c_0}  \beta(\beta +1) \intO  
(\eps+z_\eps)^{ \beta -1} \alpha_\eps (w_\eps)    | D z_\eps  |^2 dx\ dt
 \\
\leq 
  \int_\Omega  (\eps+z_\eps(0))^{\beta+1}\,dx
+ 2 (\beta+1) \intO D u_\eps \cdot D H_\eps  (x,t,Du_\eps ) \  (\eps + w_\eps )^\beta dx\ dt 
+
c  \intO      \alpha_\eps (w_\eps)  w    (\eps+z_\eps)^{ \beta  }   \ dx \ dt \,,
 \end{multline*}
where the novelty  with respect  \eqref{main33} relies on the last integral. 
Thus it is possible to   achieve the results of Theorem \ref{mainintro1} and \ref{mainintro2} only if we are able to deal with the last term, that requires in turn $1<p\leq 2$. 
\end{rem}

\begin{rem}
The papers \cite{Alikakos1,Alikakos2,Engler} addressed the preservation of Sobolev regularity as in Theorem \ref{conserv} of a similar quasilinear parabolic model with $f=0$, using an integral Bernstein argument. In particular, concerning the geometry of the domain, the paper \cite{Alikakos1} pointed out the significance of the convexity of the ambient space by contructing an explicit domain for which the gradient of solutions to the (model) heat equation with bounded initial gradient can be made arbitrarily large. Note that this is not in contrast with the previous remark and our statement of Theorem \ref{conserv}: indeed, in the non-convex case, the constant of the Sobolev a priori estimate depends on the domain, see e.g. the statement of Corollary 2.4 of \cite{PorrCCM} for the case of elliptic equations with zero-th order terms.
\end{rem}

We conclude the paper with the elliptic case and some comments to compare the estimates of Theorem \ref{mainapp} with the known results:
\begin{proof}[Sketch of the proof of Theorem \ref{mainapp}]
The proof goes along the same steps of the parabolic results. We summarize here the main changes. First, one derives the following integral inequality:
\begin{multline*}	
   \int_\Omega   2 \lambda w_\eps   (\eps+w_\eps)^{\beta }\,dx 
+   2 c_0 (\beta+1)\int_\Omega (\eps + w_\eps )^{ \beta }  \alpha_\eps (w_\eps) |D^2u_\eps|^2 dx\  
\\
+    { c_0}  \beta(\beta +1) \int_\Omega  
(\eps+w_\eps)^{ \beta -1} \alpha_\eps (w_\eps)    | D w_\eps  |^2 dx\ 
\leq 
 2 (\beta+1) \int_\Omega D u_\eps \cdot D H_\eps  (x,t,Du_\eps ) \  (\eps + w_\eps )^\beta dx\  \,.
 \end{multline*}
 This follows from Proposition 6.1 in \cite{LPcpde} and testing the $p$-Bochner identity against $(\eps+w_\eps)^\beta$. The same steps of the parabolic proof lead to the following inequality
 \begin{multline}\label{main3ell}
   \int_\Omega  2\lambda w_\eps    (\eps+w_\eps)^{\beta}\,dx 
+  c_0 (\beta+1)  \int_\Omega (\eps + w_\eps )^{ \beta }  \alpha_\eps (w_\eps) |D^2u_\eps|^2 dx\  
+    { c_0} \frac{ \beta(\beta +1)}2 \int_\Omega  
(\eps+w_\eps)^{ \beta -1} \alpha_\eps (w_\eps)    | D w_\eps  |^2 dx\ 
   \\
\leq 
 c_{1,\beta}   \int_\Omega |f_\eps (x,t) |^2  \frac{(\eps + w_\eps )^\beta }{ \alpha_\eps (w_\eps) } dx\  
+  c_{2,\beta}       \int_\Omega        \  \frac{(\eps + w_\eps )^{ {\gamma}   +\beta } }{\alpha_\eps (w_\eps)} dx\   
  \,,
  \end{multline}
where 
$$
 { c_0}= \min\left\{ 1,     i_\alpha+1   \right\}\,, \qquad   c_{1, \beta}   = 2 \frac{ (\beta+1) \big ( {N}  + \beta  \big)}{ c_0} 
  \qquad \mbox{ and } \qquad
  c_{2, \beta}  = \frac{N}{ c_0}  \frac{\gamma^2  (\beta+1) }{(\gamma +2\beta)^2 }  \,.
$$
From this point the proof is similar to the previous results, the only difference being the application of the (stationary) Sobolev inequality. We set in \eqref{main3ell}
$$
\beta= 2 	\mu - \frac{p}2 \qquad  \text{with } 
\qquad \mu= \frac{(N-2)(p-1)m}{4 (N-m)}
\,, 
$$
and  $m>m_{p,ell}$, this choice ensures that
\[
\beta>0\quad  \left(\text{i.e. }\mu>\frac{p}{4}\right)\,. 
\]
Thus, after applying Sobolev and H\"older inequalities, we deduce,
 \begin{equation*}
\begin{array}{c}\dys 
    \| \eps+ w_\eps \|^{2 \mu  }_{L^{\mu 2^*} ( \Omega)}
\leq 
c \bigg(
  1 
+      \| f_\eps (x,t) \|^2_{L^m ( \Omega)}   \|  \eps+ w_\eps\|^{2\mu- ( p-1)}_{L^{  \frac{m[2\mu- ( p-1)]}{m-2}} ( \Omega) }   \ 
+   \|          \eps+ w_\eps \|^{2 \mu + \gamma - (p- 1) }_{L^{2^*  \mu  } ( \Omega) }   +    \|   \eps+ w_\eps\|^{ \mu   }_{L^{  \mu } ( \Omega) } 
\bigg) \,.
 \end{array}
 \end{equation*}
Recalling    that $\frac{m[2\mu- ( p-1)]}{m-2}= \mu 2^* $,   
 \begin{equation*} 
\begin{array}{c}\dys 
    \int_\Omega  2 \lambda w_\eps  (\eps+w_\eps)^{2 	\mu - \frac{p}2 }\,dx 
+
    \| \eps+ w_\eps \|^{2 \mu }_{L^{\mu 2^*} ( \Omega)}
    +  \frac12 c_0 (\beta+1)  \int_\Omega (\eps + w_\eps )^{ \beta }  \alpha_\eps (w_\eps) |D^2u_\eps|^2 dx\  
\\[1.5 ex] \dys 
\leq 
c \bigg(
  1+ 
      \| f_\eps (x,t) \|^{\frac{4 \mu}{p-1}}_{L^m ( \Omega)}  
      +   \|          \eps+ w_\eps\|^{2 \mu + \gamma - (p- 1) }_{L^{2^* \mu  } ( \Omega) }  
      +    \|   \eps+ w_\eps\|^{ \mu   }_{L^{  \mu } ( \Omega) }  \bigg).
 \end{array}
 \end{equation*}
The proof concludes since the last two terms can be absorbed in the left hand side since $0\leq \gamma < p-1$. 
Moreover \eqref{soell1} follows by plugging \eqref{ell1} into \eqref{main3ell}. Furthermore,  \eqref{ell2}--\eqref{soell2} can be deduced as a limit case from \eqref{ell1}--\eqref{soell1}. \\
Finally, in order to deduce the $L^\infty$ bound for $p\geq 2$, one has to choose $(w_\eps -k)_+^\mu$, $\mu >0$ as test function, and follow the proof of Theorem \ref{mainintro3}, with the use of Sobolev inequality instead of Gagliardo Nirenebrg one. 

\end{proof}
\begin{rem}
Estimate \eqref{soell1} agrees with Theorem 2.3 in \cite{CianchiMazyaARMA2} in the limit $m\to2$ and $\gamma=0$, and extends it to general $f\in L^m(\Omega)$ and lower order sublinear terms. Estimates \eqref{e1}-\eqref{e2} recover with a different proof those in Theorem 4.3 of \cite{CianchiMazyaJEMS}.
\end{rem}

\section{Existence results}\label{exi}

In this section we obtain,  as a byproduct of the results stated in Section \ref{sec;assres}, the existence of a solution to 
\begin{equation}\label{parap}
\begin{cases}
\partial_t u-\mathrm{div}(\alpha(|Du|^2) Du) =|Du|^\gamma  + f (x,t)&\quad\text{ on }Q_T,\\
\partial_\nu u=0&\quad\text{ on } \partial\Omega\times(0,T),\\
u(x,0)=u_0(x)&\quad\text{ on } \Omega\ , 
\end{cases}
\end{equation}
under the hypotheses of Theorems \ref{mainintro1}, \ref{mainintro1} or \ref{mainintro3}.

\begin{cor}
 Let $p>1$, $\Omega $ be convex,  
$ 0\leq \gamma   <  \ell$
and 
$$
f  \in L^{m} (Q_T)\, , \quad \mbox{  with } \quad m>m_p \,, \quad \mbox{ where  } \quad m_p := \max\left\{ 2, \frac{Np+4}{N(p-1)+2} \right\}\,.
$$ 
If $u_0 \in W^{1, \rho} (\Omega)$,  with $ \rho = N\frac{(p-1)m-(p-2)}{N+2-m}$, 
then there exists a weak solution to \eqref{parap}. 
Moreover 
\begin{itemize} 
\item[a)] if $m_p<m<N+2$, then estimates \eqref{e1}--\eqref{so} hold true; 
\item[b)] if $ m=N+2$, $u_0 \in W^{1,\rho} (\Omega)$ for any $\rho\geq 1$, then estimates \eqref{quasilinfty}--\eqref{so2} hold true; 
\item[c)] if $ m>N+2$, $u_0 \in W^{1,\infty} (\Omega)$ and $p\geq 2$,  then $|D u|$ is bounded.  
\end{itemize} 
 \end{cor} 

Here we sketch   the proof of the above result: for a more complete and detailed 
 proof  one can follow \cite{PorzioDCDS}.

\proof[Sketch of the proof] 

Let us consider the sequence $u_\eps$ of solutions   to \eqref{approxp}  with   $ u_\eps \in L^p (0,T; W^{1,p} (\Omega)) \cap C^0 ([0,T]; L^2 (\Omega))$ such that $\partial_t u_\eps \in L^{p'} (0,T; W^{-1,p'} (\Omega))+L^1 (Q_T)$, whose existence follows by  classical results on quasilinear parabolic equations, see e.g. Theorem 13.24 in \cite{Lieberman}. More precisely, $u_\eps$ satisfies  
\begin{multline*}
\iint_{Q_T}  \partial_t  u_\eps (x,t) \psi (x,t) dx +
\iint_{Q_T} \alpha (\eps+|Du_\eps|^2) D u_\eps \cdot D \psi (x,t) dx dt 
+ \iint_{Q_T} |D u_\eps|^\gamma  \psi (x,t) dx dt 
= 
\iint_{Q_T} f (x,t) \psi (x,t)  dx dt \,,
\end{multline*}
for any $ \psi \in L^p (0,T; W^{1,p} (\Omega)) \cap L^\infty (Q_T)$.
 
We want to show that we can pass to the limit in its weak formulation
 and get a solution to \eqref{parap}.  
  Let us first recall that by a weak solution $u$ to problem \eqref{parap} we mean a function in
$C([0,T]; L^1 (\Omega)) $ such that $|Du|^{p-1} \in L^\mu (Q_T) $ with $\mu = \max\{\frac{\gamma}{p-1},1\}$  that satisfies 
\begin{multline*}
\int_{\Omega} u (x,t) \varphi (x,t) dx +
\iint_{Q_T} \alpha (|Du|^2) D u \cdot D \varphi (x,t) dx dt 
+ \iint_{Q_T} |D u|^\gamma  \varphi (x,t) dx dt  
\\
= \iint_{Q_T} f (x,t) \varphi (x,t)  dx dt 
+ 
\iint_{Q_T}u  (x,t) \partial_t\varphi (x,t)  dx dt 
+
\int_{\Omega} u (x,0) \varphi (x,0)  dx dt \,,
\end{multline*}
for any $  \varphi \in C^1 (Q_T)$. 
  
  \smallskip 
  
{\bf  a) }
We first observe that   Theorem \ref{mainintro1} guarantees estimates \eqref{e1}--\eqref{so}. Moreover, choosing   
$\psi_\eps (x,t)= \frac1\delta T_\delta (u_\eps)$, with $\delta>0$ we deduce, after sending $\delta$ to $0$,  a $L^{\infty} (0,T;L^1 (\Omega))$ bound on $u_\eps$. 

 In particular there exists a function $u\in L^q (0,T; W^{1,q} (\Omega))\cap L^{\infty} (0,T; W^{1,\rho} (\Omega))$ such that, up to subsequences (not relabeled), 
  $u_\eps $ converges to $u$, as $\eps$ vanishes, weakly in  $L^q (0,T; W^{1,q} (\Omega))$, strongly in $L^q (Q_T)$ and a.e. in $Q_T$.  
Furthermore we have that 
 $ \partial_t u_\eps-\Delta_p u_\eps $ is bounded in $L^1 (Q_T)$ so that by Theorem 3.3 in \cite{BDGOjfa} one obtains the a.e. convergence of gradients. Thus, thanks to Vitali's Theorem we have compactness of $|D u_\eps|^{p-1}$ in $L^{\frac{\gamma}{p-1}} (Q_T)\cap L^1 (Q_T)$. 
  
This, jointly with  \eqref{e1} and \eqref{e2}, allow us to pass to the limit both in the principal part of the operator and in the gradient term, by using  Vitali's theorem.

 \smallskip 

 {\bf  b)} and   {\bf  c)} follow in the same way, applying Theorem    \ref{mainintro2} and \ref{mainintro3} respectively instead of Theorem    \ref{mainintro1}. 
  \qed

 \medskip

A similar result holds for the stationary problem: consider now the elliptic problem 
\begin{equation}\label{ppell1}
\begin{cases} 
\lambda u- \text{div}  \big( \alpha (|Du|^2)  D u\big)  =|Du|^\gamma + f(x) \qquad & \mbox{in }  \Omega  \,,\\
\partial_\nu u = 0  & \mbox{on } \partial \Omega.  \, \end{cases}
\end{equation}
Then we have the following existence results. 
\begin{cor} 
Let $p>1$, $\Omega $ be convex,  
$ 0\leq \gamma   <  p-1
$
and 
$$
f  \in L^{m} ( \Omega)\, , \quad \mbox{  with } \quad m>m_p \,, \quad \mbox{ where  } \quad m_{p,\mathrm{ell}} := \max\left\{ 2, \frac{Np}{N(p-1)-(p-2)} \right\}\,.
$$ 
Then there exists a weak solution to \eqref{ppell1} such that 

\begin{itemize} 
\item if $m_{p,\mathrm{ell}}<m<N$, then \eqref{ell1} and  \eqref{soell1} hold true; 
\item if $m=N$, then  \eqref{ell2} and  \eqref{soell2} hold true; 
\item if $m>N $ and  $p\geq 2$, then   $|Du| $ is bounded.
\end{itemize}
\end{cor}

\proof The existence of a weak solution follows exactly as in 
\cite[Remark 4.8]{CGL}, while the estimates are a consequence of Theorem \ref{mainapp}. \qed

\medskip

\end{document}